\documentclass[oneside,english]{amsart}
\usepackage[T1]{fontenc}
\usepackage[latin9]{inputenc}
\usepackage{verbatim}
\usepackage{refstyle}
\usepackage{units}
\usepackage{mathrsfs}
\usepackage{amsthm}
\usepackage{amstext}
\usepackage{amssymb}
\usepackage{esint}

\makeatletter

\AtBeginDocument{\providecommand\secref[1]{\ref{sec:#1}}}
\AtBeginDocument{\providecommand\propref[1]{\ref{prop:#1}}}
\RS@ifundefined{subref}
  {\def\RSsubtxt{section~}\newref{sub}{name = \RSsubtxt}}
  {}
\RS@ifundefined{thmref}
  {\def\RSthmtxt{theorem~}\newref{thm}{name = \RSthmtxt}}
  {}
\RS@ifundefined{lemref}
  {\def\RSlemtxt{lemma~}\newref{lem}{name = \RSlemtxt}}
  {}

\numberwithin{equation}{section}
\numberwithin{figure}{section}
\theoremstyle{plain}
\newtheorem{thm}{\protect\theoremname}[section]
  \theoremstyle{plain}
  \newtheorem{conjecture}[thm]{\protect\conjecturename}
  \theoremstyle{remark}
  \newtheorem*{acknowledgement*}{\protect\acknowledgementname}
  \theoremstyle{remark}
  \newtheorem*{notation*}{\protect\notationname}
 \theoremstyle{definition}
 \newtheorem*{defn*}{\protect\definitionname}
  \theoremstyle{plain}
  \newtheorem{prop}[thm]{\protect\propositionname}
  \theoremstyle{plain}
  \newtheorem{cor}[thm]{\protect\corollaryname}
  \theoremstyle{plain}
  \newtheorem{lem}[thm]{\protect\lemmaname}
  \theoremstyle{remark}
  \newtheorem*{rem*}{\protect\remarkname}

\IfFileExists{lmodern.sty}
 {\usepackage{lmodern}}{}

\title [The Fourth Moment of Dirichlet L-Functions]{The Fourth Moment of Dirichlet L-Functions for the Rational Function Field}
\date{\today}

\makeatother

\usepackage{babel}
  \providecommand{\acknowledgementname}{Acknowledgement}
  \providecommand{\conjecturename}{Conjecture}
  \providecommand{\corollaryname}{Corollary}
  \providecommand{\definitionname}{Definition}
  \providecommand{\lemmaname}{Lemma}
  \providecommand{\notationname}{Notation}
  \providecommand{\propositionname}{Proposition}
  \providecommand{\remarkname}{Remark}
\providecommand{\theoremname}{Theorem}

\begin{document}

\author{Nattalie Tamam }

\address{Raymond and Beverly Sackler School of Mathematical Sciences, Tel
Aviv University, Tel Aviv 69978, Israel}

\email{nattalie.com@gmail.com}

\thanks{Partially supported by the Israel Science Foundation (grant No. 1083/10).}
\begin{abstract}
We study the moments of the Dirichlet L-function when defined over
the polynomial ring over finite fields. We find an asymptotic formula
to the fourth moment of the central value of Dirichlet L functions
in this context. We also find a lower bound to the $2k$th moment
of these L-functions. 
\end{abstract}
\maketitle

\section{Introduction }

This work deals with function field analogues of recent studies concerning
moments of central values $L\left(\frac{1}{2},\chi\right)$ of Dirichlet
L-functions, where it has been conjectured that as $\chi$ varies
over all (primitive) Dirichlet characters modulo $Q$, the $2k$-th
moment of $L(\frac{1}{2},\chi)$ is asymptotically equal to 
\[
C_{k}Q\left(\log Q\right)^{k^{2}},\quad Q\to\infty
\]
 for a positive constant $C_{k}$. An exact form for $C_{k}$ was
conjectured by Keating and Snaith \cite{key-6} using Random Matrix
Theory. The similarity between the statistics of the zeros of the
Riemann zeta function and the eigenvalues of random unitary matrices
chosen uniformly with respect to Haar measure was first observed by
Montgomery and Dyson \cite{key-30,key-32,key-31}. Keating and Snaith
\cite{key-6} introduced a random matrix model for the study of L-functions.
They suggested that the value distribution of the L-functions on the
critical line is related to the characteristic polynomials of random
unitary matrices. Earlier, Katz and Sarnak \cite{key-33} conjectured
that statistics of low lying zeros of families of L-functions coincide
with the distribution of low-lying eigenvalues of the classical compact
groups. They divided the L-functions into symmetry types families
and found an appropriate conjecture to each one. Since the family
of all Dirichlet L-functions is believed to be unitary (see the recent
results of Katz \cite{key-34}), Keating and Snaith conjectured an
exact formula for the constant $C_{k}$ using moments of characteristic
polynomials of random unitary matrices, for which Conjecture \ref{conj moments}
below is the equivalent for the ring of polynomials.

We work with the finite field $\mathbb{F}_{q}$, where $p$ a prime,
$q=p^{n}$ its power, and the polynomial ring $\mathcal{A}=\mathbb{F}_{q}\left[x\right]$
over it. For a nonzero polynomial $Q\in\mathcal{A}$ set $|Q|=q^{\deg Q}$.\linebreak{}
 For a Dirichlet character $\chi$ modulo $Q$ denote by $L(s,\chi)$
the associated Dirichlet L-function (see e.g. \cite{key-4}).

We first formulate a direct analog in the polynomial ring of the Keating-Snaith
Conjecture.
\begin{conjecture}
\label{conj moments}For polynomial $Q$ and $k\in\mathbb{N}$
\[
\frac{1}{\phi^{*}\left(Q\right)}\underset{\chi\pmod Q}{\mathrm{\sum}^{*}}\left|L\left(\frac{1}{2},\chi\right)\right|^{2k}\sim a\left(k\right)\frac{G^{2}\left(k+1\right)}{G\left(2k+1\right)}\prod_{P\mid Q}\left(\sum_{m\geq0}\frac{d_{k}\left(P^{m}\right)^{2}}{\left|P\right|^{m}}\right)^{-1}\left(\deg Q\right)^{k^{2}}
\]
as $\deg Q\rightarrow\infty$, when $G\left(z\right)$ is the Barnes
G-function, $d_{k}\left(N\right)$ is the number of ways to represent
$N$ as a product of $k$ factors when $k\in\mathbb{N}$, and $\phi^{*}\left(Q\right)$
is the number of primitive characters modulo $Q$, and $\mathrm{\sum}^{*}$
denotes summation over all primitive characters modulo $Q$. Also
\begin{eqnarray*}
a(k) & = & \prod_{P}\left(\left(1-\frac{1}{\left|P\right|}\right)^{k^{2}}\sum_{m\geq0}\frac{d_{k}\left(P^{m}\right)}{\left|P\right|^{m}}\right).
\end{eqnarray*}

\end{conjecture}
In this paper, the main proofs are under the assumption that $Q$
is an irreducible polynomial, and the summations are over all nontrivial
characters. For an irreducible polynomial all nontrivial characters
are primitive. Thus, these summations are over all primitive characters\emph{.
}We will calculate the first and second moment of $L(\frac{1}{2},\chi)$,
and obtain an immediate conclusion about their non-vanishing. Our
main result concerns the $4$-th moment, and is analogous to the result
by Heath-Brown\cite{key-1}: 
\begin{thm}
\label{thm 4th}For all irreducible $Q\in\mathcal{A}$
\[
\frac{1}{\phi(Q)}\sum_{\chi\neq\chi_{0}}\left|L\left(\frac{1}{2},\chi\right)\right|^{4}=\frac{q-1}{12q}\left(\deg Q\right)^{4}+O\left(\left(\deg Q\right)^{3}\right).
\]

\end{thm}
A direct calculation will show 
\[
a_{2}=\zeta\left(2\right)^{-1}=1-\frac{1}{q},\qquad\frac{G^{2}\left(3\right)}{G\left(5\right)}=\frac{1}{12}
\]
\[
\prod_{P\mid Q}\left(\sum_{m\geq0}\frac{d_{2}\left(P^{m}\right)^{2}}{\left|P\right|^{m}}\right)^{-1}=1+O\left(\frac{1}{\left|Q\right|}\right).
\]
 Hence our Theorem \ref{thm 4th} is consistent with Conjecture $\ref{conj moments}$.

For $k=2$, the formula 
\[
\frac{1}{\phi^{*}\left(Q\right)}\underset{\chi\pmod Q}{\mathrm{\sum}^{*}}\left|L\left(\frac{1}{2},\chi\right)\right|^{4}\sim\frac{1}{2\pi^{2}}\prod_{p\mid Q}\frac{\left(1-p^{-1}\right)^{3}}{\left(1+p^{-1}\right)}\left(\log Q\right)^{4}
\]
 for almost all large $Q$, was proved by Heath-Brown \cite{key-1},
in 1981. In 2005, K. Soundararajan, \cite{key-2}, improved this result
by showing that it holds for all large $Q$. Later on Young \cite{key-26}
proved that for a prime modulus $Q$ the main term in the fourth moment
is a polynomial in $\log Q$, with a power saving 
\[
\frac{1}{\phi^{*}\left(Q\right)}\underset{\chi\pmod Q}{\mathrm{\sum}^{*}}\left|L\left(\frac{1}{2},\chi\right)\right|^{4}=\sum_{i=0}^{4}c_{i}\left(\log Q\right)^{i}+O\left(Q^{-\frac{5}{512}+\epsilon}\right),
\]
for certain computable absolute constants $c_{i}$. For $k>2$ there
are no proven asymptotic results.

Our second result is an analog of the general lower bound of Rudnick
and Soundararajan \cite{key-11}, for the $2k$-th moment
\begin{thm}
Let $k$ be a fixed natural number. Then for all irreducible polynomial
$Q$, with a sufficiently large degree 
\[
\underset{\chi\neq\chi_{0}}{\sum_{\chi\pmod Q}}\left|L\left(\frac{1}{2},\chi\right)\right|^{2k}\gg_{k}\left|Q\right|\left(\deg Q\right)^{k^{2}}.
\]

\end{thm}
Note that a lower bound was recently proved over the integers for
rational $k\in\left(0,1\right)$ by V. Chandee and X. Li, \cite{key-16},
and for $1<k\in\mathbb{R}$, large $Q$ (not necessarily prime) by
Radziwi\l{}\l{} and Soundararajan, \cite{key-15}.
\begin{acknowledgement*}
I would like to thank Zeev Rudnick for suggesting the problems studied
in the present work and many helpful discussions and suggestions in
the course of research and writing the paper. Also, I thank the referee
for several comments on an early version of the paper. This work is
part of the author\textquoteright{}s M. Sc. thesis written under the
supervision of Zeev Rudnick at Tel-Aviv University. 
\end{acknowledgement*}

\section{An Outline of the Proof of Theorem \ref{thm 4th}}

One of the differences between the Dirichlet L function and its function
field analog $L\left(s,\chi\right)$, where $\chi$ is a character
modulo $Q$, is that in the latter case, the L-function is a polynomial
in $q^{-s}$ of degree $\mathfrak{D}:=\deg Q-1$. This will be recalled
later on. When evaluating the fourth moment, we work with the squared
L function, which brings us to a polynomial of degree $2\mathfrak{D}$.
By using the functional equation we can reduce our L function into
a function with a smaller degree, $\mathfrak{D}$ (section \ref{sec:Expressing L as a short sum}).
This will simplify the calculation. Now, the mean value of the summation
over all characters mod $Q$ is
\[
\frac{1}{\phi\left(Q\right)}\sum_{\chi\neq\chi_{0}}\left|L\left(\frac{1}{2},\chi\right)\right|^{4}\sim\frac{4}{\phi\left(Q\right)}\sum_{\chi\pmod Q}\underset{\deg AB,\deg CD\leq\mathfrak{D}}{\sum}\frac{\chi\left(AC\right)\bar{\chi}\left(BD\right)}{\left|ABCD\right|^{\nicefrac{1}{2}}}
\]
when the remainder term is bounded by $O\left(\mathfrak{D}^{3}\right)$.
The orthogonality relation for characters mod $Q$ implies that only
the terms $AC\equiv BD$ remain. We split this sum into diagonal and
off-diagonal terms
\begin{eqnarray*}
\underset{\deg AB,\deg CD\leq\mathfrak{D}}{\sum_{AC\equiv BD\pmod Q}}\left|ABCD\right|^{-\nicefrac{1}{2}} & = & \underset{\deg AB,\deg CD\leq\mathfrak{D}}{\sum_{AC=BD\pmod Q}}\left|ABCD\right|^{-\nicefrac{1}{2}}\\
 &  & \qquad\qquad+\underset{\deg AB,\deg CD\leq\mathfrak{D}}{\sum_{AC\equiv BD\pmod Q,\: AC\neq BD}}\left|ABCD\right|^{-\nicefrac{1}{2}}.
\end{eqnarray*}
We can calculate the diagonal term 
\[
\underset{\deg AB,\deg CD\leq\mathfrak{D}}{\sum_{AC=BD\pmod Q}}\left|ABCD\right|^{-\nicefrac{1}{2}}=\frac{q-1}{48q}\mathfrak{D}^{4}+O\left(\mathfrak{D}^{3}\right)
\]
by comparison of series coefficients. To estimate the off-diagonal
term, we divide the terms $\deg AB,\deg CD\leq\mathfrak{D}$ into
dyadic blocks. We will use a bound for divisor sums over arithmetic
progressions 
\[
\underset{N\equiv A\pmod Q}{\underset{\deg N\leq x}{\sum}}d\left(N\right)\ll\frac{q^{x}x}{\phi\left(Q\right)},
\]
which is the polynomial analog for a Theorem of P. Shiu \cite{key-3}
that we will prove in \ref{sec:Divisor-Sums}. For each block we have
\[
\underset{AC\neq BD}{\underset{AC\equiv BD\pmod Q}{\underset{Z_{2}-2\leq\deg CD\leq Z_{2}}{\underset{Z_{1}-2\leq\deg AB\leq Z_{1}}{\sum}}}}1\ll\frac{q^{Z_{1}+Z_{2}}\left(Z_{1}Z_{2}\right)^{3}}{\left|Q\right|}.
\]
Therefore, the off-diagonal term is bounded by $O\left(\mathfrak{D}^{3}\right)$. 

The full version of the above proof is in sections \ref{sec:The-Fourth-Moment}.
Along the way, we shall require function field analogues of results
in sieve theory (section \ref{sec:The-Selberg-Sieve}).

\section{Background}

A survey of number theory over function fi{}elds can be found in \cite{key-4}.

\subsection{The Ring of Polynomials Over a Finite Field\label{sub:background.1}}
\begin{notation*}
$\mathcal{A}$ will denote the polynomial ring over $\mathbb{F}$.
We will use upper-case letters such as $A,\: B,\: C,\: D,\: N,\: M$
to denote monic polynomials in $\mathcal{A}$, the letter $P$ to
denote an irreducible monic polynomial in $\mathcal{A}$, and lower
case letters such as $x,y,z$ to denote integers. 
\end{notation*}
In all that follow $\mathbb{F}$ will denote a finite field with $q$
elements, when $q=p^{n}$ for some prime integer $p$ and $n\geq1$. 

Every element $N\in\mathcal{A}$ has the form 
\[
N=a_{n}T^{n}+a_{n-1}T^{n-1}+\ldots+a_{0}
\]
 if $a_{n}\neq0$ we say that $N$ has degree $n$, notationally $\deg N=n$.
If $a_{n}=1$ we say that $N$ is a monic polynomial. We will only
work with monic polynomials. 

If $0\neq N\in\mathcal{A}$, $\mathcal{A}/N\mathcal{A}$ is a finite
ring with $q^{\deg N}$ elements, $\left|4\right|$. The absolute
value of a polynomial $N\in\mathcal{A}$ is defined by $\left|N\right|=q^{\deg N}$
if $N\neq0$, or otherwise by $\left|0\right|=0$. 

We denote by $\phi$ the Euler totient function for polynomial: 
\[
\phi(N)=\#\left\{ M\in\mathcal{A}/N\mathcal{A}\mid\left(M,N\right)=1\right\} =\left|N\right|\prod_{P\mid N}\left(1-\frac{1}{\left|P\right|}\right).
\]

We denote by $\pi\left(n\right)$ the number of monic irreducible
polynomials of degree $n$. The prime polynomial theorem says that
\[
\sum_{\deg N=n}\Lambda\left(N\right)=q^{n},\qquad\Lambda\left(N\right)=\begin{cases}
\deg N; & \mbox{if }N=P^{k},\: P\mbox{ prime, }k\geq1\\
0; & \mbox{else}
\end{cases}.
\]
Therefore we have 
\[
n\pi\left(n\right)\leq\sum_{d\mid n}d\pi\left(d\right)=q^{n},
\]
and in particular $\pi\left(n\right)\leq\frac{q^{n}}{n}$. Also, the
prime polynomial theorem says that \linebreak{}
$\pi\left(n\right)=\frac{q^{n}}{n}+O\left(\frac{q^{\nicefrac{n}{2}}}{n}\right)$
(by the Möbius inversion formula).

In \secref{Divisor-Sums} we will use the function $d(N)$ to denote
the divisor function for polynomials $d(N)=\sum_{d\mid N}1$, the
sum is over monic polynomials. It can be calculated, that for an integer
$x\geq1$, 
\begin{eqnarray*}
\sum_{\deg N\leq x}d(N) & = & \frac{\left(x+1\right)q^{x+1}-\frac{q^{x+1}-1}{q-1}}{q-1}
\end{eqnarray*}
thus we have the bound

\[
\sum_{\deg N\leq x}d(N)\ll xq^{x}.
\]
This bound was also proved in the integer ring setting in \cite{key-7}%
. 

We let $p_{+}(N)$ and $p_{-}(N)$ denote the greatest and least degree
of prime factor of $N$ respectively, while $\Omega(N)$ is the number
of prime factors of $N$, taking into account the multiplicativity. 
\begin{defn*}
For $x,y,z\geq0,$ and $A,K$ monic polynomials we define

\[
\Psi\left(x,z\right)=\underset{p_{+}(N)\leq z}{\underset{\deg N\leq x}{\sum}}1,\quad\textrm{and}\quad\Phi\left(x,z;K,A\right)=\underset{p_{-}(N)>z}{\underset{N\equiv A(\mbox{mod}\: K)}{\underset{\deg N<x}{\sum}}}1.
\]

\end{defn*}
$\Psi(x,z)$ is also known as the de Bruijn function, denoting the
number of $z$-smooth polynomials of a degree less or equal to $x$.
A polynomial is called $z$-smooth if none of its prime factors degree
is greater than $z$. It is known that the de Bruijn function satisfies
$\Psi(x,z)=q^{x}\rho(\frac{x}{z})+O\left(\frac{q^{x}}{z}\right)$
where $\rho$ is the Dickman function, for which it is known that
$\rho(u)\approx u^{-u}$, \cite{key-9}. 

In section \secref{The-Selberg-Sieve} we will use:
\begin{defn*}
The Möbius function for polynomials, $\mu$, is defined by 
\[
\mu(N)=\begin{cases}
\left(-1\right)^{k}; & \mbox{if }N=P_{1}\cdots P_{k}\mbox{ for }1\leq i<j\leq k:\: P_{i}\neq P_{j}\\
0; & \mbox{else}
\end{cases}
\]

\end{defn*}
and satisfies
\[
\sum_{D\mid N}\mu(D)=\begin{cases}
1; & N=1\\
0; & else
\end{cases}.
\]

\subsection{Dirichlet L-Functions}

Let $Q\in\mathcal{A}$, and denote $\mathfrak{D}=\deg Q-1$. We will
mostly assume that $Q$ is irreducible.
\begin{defn*}
A Dirichlet character modulo $Q$ is a function $\chi$ from $\mathcal{A\longrightarrow}\mathbb{C}$
such that

1. $\chi\left(A+BQ\right)=\chi\left(A\right)$ for all $A,B\in\mathcal{A}$.

2. $\chi\left(A\right)\chi\left(B\right)=\chi\left(AB\right)$ for
all $A,B\in\mathcal{A}$.

3. $\chi\left(A\right)\neq0$ if and only if $\left(A,Q\right)=1$.

A character $\chi$ modulo $Q$ is called primitive if its minimal
period is $Q$ (for a prime $Q$ all characters are primitive). 
\end{defn*}
Dirichlet characters satisfy the following orthogonal relations: if
$\chi,\psi$ Dirichlet characters modulo $Q$ and $A,B\in\mathcal{A}$
are relatively prime to $Q$, then
\begin{enumerate}
\item $\sum_{A\in\mathcal{A}}\chi\left(A\right)\bar{\psi}\left(A\right)=\phi\left(Q\right)\delta\left(\chi,\psi\right)$.
\item $\sum_{\chi\in\hat{\mathcal{A}}}\chi\left(A\right)\bar{\chi}\left(B\right)=\phi(Q)\delta\left(A,B\right)$.
\end{enumerate}
The Dirichlet L function is defined for $\Re\left(s\right)>1$ by
\[
L\left(s,\chi\right)=\sum_{N\mbox{ monic}}\frac{\chi\left(N\right)}{\left|N\right|^{s}}.
\]
For each $\chi$ non-trivial Dirichlet character modulo $Q$ we may
denote $u=q^{-s}$ and rewrite $L\left(s,\chi\right)$ as 
\[
L^{*}\left(u,\chi\right)=\sum_{N\mbox{ monic}}\chi\left(N\right)u^{\deg N}=\sum_{n=0}^{\infty}L_{n}\left(\chi\right)u^{n}
\]
 when $L_{n}\left(\chi\right)=\sum_{\deg N=n}\chi\left(N\right)$.
Now, if $m>\mathfrak{D}$, for each monic $M\in\mathcal{A}$ such
that $\deg M=m$, we can write uniquely $M=SQ+R$ when $\deg R\leq\mathfrak{D}$
or $R=0$. Since $\chi$ is periodic modulo $Q$ 
\[
\sum_{\deg M=m}\frac{\chi\left(M\right)}{\left|M\right|^{s}}=q^{-ms}\sum_{\deg R\leq\mathfrak{D}}\chi\left(R\right)=0.
\]
Hence, we have 
\[
L^{*}\left(u,\chi\right)=\sum_{n=0}^{\mathfrak{D}}L_{n}\left(\chi\right)u^{n},
\]
 i.e. $L^{*}\left(u,\chi\right)$ is a polynomial in $u$ of degree
$\mathfrak{D}$. 
\begin{defn*}
$\chi$ is an ``even'' character if $\chi\left(cN\right)=\chi\left(N\right)$
for all $c\in\mathbb{F}^{*},N\in\mathbb{F}[x]$. 
\end{defn*}
If we define 

\[
\lambda_{\chi}=\begin{cases}
1, & \chi\mbox{ "even"}\\
0, & \mbox{else}
\end{cases}\;,\;\Lambda\left(u,\chi\right):=\left(1-\lambda_{\chi}u\right)^{-1}L^{*}\left(u,\chi\right)
\]
then by \cite{key-4}, the \textquotedblleft{}completed\textquotedblright{}
L-function $\Lambda\left(u,\chi\right)$ is a polynomial in $u$ of
degree $\mathfrak{D}$ if $\chi$ is an ``even'' character, and
of degree $\mathfrak{D}-1$ otherwise, which satisfies the functional
equation

\[
\Lambda\left(u,\chi\right)=\epsilon\left(\chi\right)\left(q^{\nicefrac{1}{2}}u\right)^{\mathfrak{D}-\lambda_{\chi}}\Lambda\left(\frac{1}{qu},\bar{\chi}\right)
\]
when $\left|\epsilon\left(\chi\right)\right|=1$. There is a simple
proof for the case that $\chi$ is not ``even'' in \cite{key-5},
and a complete proof in \cite{key-4}.

\section{The First and Second Moments}

During the study of the Dirichlet L-functions moments for polynomials
basic results are the first and second moment. In this section the
direct calculations of the first and second moments are brought for
completeness. Also, these results give an immediate conclusion about
the non-vanishing of $L\left(\frac{1}{2},\chi\right)$. In the number
field case, it is believed that $L\left(\frac{1}{2},\chi\right)$
does not vanish. It has been proved that a positive proportion of
the L-values are non-zero (with the latest record showing $34\%$),
\cite{key-27,key-28,key-29}. By denoting 
\[
X\left(Q\right)=\left\{ \chi\neq\chi_{0}\mbox{ mod }Q\mid L\left(\frac{1}{2},\chi\right)\neq0\right\} 
\]
for an irreducible $Q\in\mathcal{A}$, we can deduce (see \ref{eq:coro, sec&first})
\[
\frac{\left|X\left(Q\right)\right|}{\left|\left\{ \chi\neq\chi_{0}\mbox{ mod }Q\right\} \right|}\gg\frac{1}{\mathfrak{D}}.
\]
This bound implies non-vanishing in a non empty set, but one of decreasing
size. The possibility of improving this bound to a positive proportion
seems intriguing.
\begin{prop}
\label{prop:first momen}Let $Q\in\mathcal{A}$ be an irreducible
polynomial. The mean value of the Dirichlet L- function is
\[
\frac{1}{\phi\left(Q\right)}\sum_{\chi\neq\chi_{0}}L\left(\frac{1}{2},\chi\right)=1+O\left(q^{-\nicefrac{\mathfrak{D}}{2}}\right).
\]
\end{prop}
\begin{proof}
We have 

\begin{eqnarray*}
\frac{1}{\phi\left(Q\right)}\sum_{\chi\neq\chi_{0}}L\left(\frac{1}{2},\chi\right) & = & \frac{1}{\phi\left(Q\right)}\sum_{\chi\neq\chi_{0}}\sum_{\deg N<\mathfrak{D}}\frac{\chi\left(N\right)}{\left|N\right|^{\nicefrac{1}{2}}}\\
 & = & \frac{1}{\phi\left(Q\right)}\sum_{\deg N<\mathfrak{D}}\left|N\right|^{-\nicefrac{1}{2}}\sum_{\chi\neq\chi_{0}}\chi\left(N\right)
\end{eqnarray*}
 we know
\[
\frac{1}{\phi\left(Q\right)}\sum_{\chi\neq\chi_{0}}\chi\left(N\right)=\begin{cases}
1-\frac{\chi_{0}\left(N\right)}{\phi\left(Q\right)}=1-\frac{1}{\phi\left(Q\right)}; & N\equiv1\pmod Q\\
-\frac{\chi_{0}\left(N\right)}{\phi\left(Q\right)}=-\frac{1}{\phi}; & \mbox{else}
\end{cases}
\]
since $\deg N<\mathfrak{D}$ we have $N\equiv1\pmod Q\iff N=1$. Hence
\begin{eqnarray*}
\frac{1}{\phi\left(Q\right)}\sum_{\chi\neq\chi_{0}}L\left(\frac{1}{2},\chi\right) & = & 1-\frac{1}{\phi\left(Q\right)}-\frac{1}{\phi\left(Q\right)}\sum_{n=1}^{\mathfrak{D}-1}q^{-\nicefrac{n}{2}}q^{n}\\
 & = & 1-\frac{1}{\phi\left(Q\right)}-\frac{1}{\phi\left(Q\right)}\frac{q^{\nicefrac{\mathfrak{D}}{2}}-q}{q^{\nicefrac{1}{2}}-1}=1+O\left(q^{-\nicefrac{\mathfrak{D}}{2}}\right),
\end{eqnarray*}
 and we get the desired result. \end{proof}
\begin{prop}
\label{prop:second moment}Let $Q\in\mathcal{A}$ be an irreducible
polynomial. The second moment of the Dirichlet L- function is 
\[
\frac{1}{\phi\left(Q\right)}\sum_{\chi\neq\chi_{0}}\left|L\left(\frac{1}{2},\chi\right)\right|^{2}=\mathfrak{D}+O\left(1\right).
\]
\end{prop}
\begin{proof}
We have 
\begin{eqnarray*}
\frac{1}{\phi(Q)}\sum_{\chi\neq\chi_{0}}\left|L\left(\frac{1}{2},\chi\right)\right|^{2} & = & \frac{1}{\phi\left(Q\right)}\sum_{\chi\neq\chi_{0}}\sum_{\deg N,\:\deg M<\mathfrak{D}}\frac{\chi\left(N\right)\bar{\chi}\left(M\right)}{\left|NM\right|^{2}}\\
 & = & \frac{1}{\phi\left(Q\right)}\sum_{\deg N,\:\deg M<\mathfrak{D}}\left|NM\right|^{-\nicefrac{1}{2}}\sum_{\chi\neq\chi_{0}}\chi\left(N\right)\bar{\chi}\left(M\right),
\end{eqnarray*}
we know
\[
\frac{1}{\phi\left(Q\right)}\sum_{\chi\neq\chi_{0}}\chi\left(N\right)\bar{\chi}\left(M\right)=\begin{cases}
1-\frac{1}{\phi\left(Q\right)}; & N\equiv M\pmod Q\\
-\frac{1}{\phi\left(Q\right)}; & \mbox{else}
\end{cases}
\]
 since $\deg N,\:\deg M<\mathfrak{D}$ we have $N\equiv M\pmod Q\iff N=M$.
Hence
\begin{eqnarray*}
\frac{1}{\phi(Q)}\sum_{\chi\neq\chi_{0}}\left|L\left(\frac{1}{2},\chi\right)\right|^{2} & = & \left(1-\frac{1}{\phi\left(Q\right)}\right)\underset{\deg N,\:\deg M<\mathfrak{D}}{\sum_{N=M}}\left|NM\right|^{-\nicefrac{1}{2}}\\
 &  & \qquad\quad-\frac{1}{\phi\left(Q\right)}\underset{\deg N,\:\deg M<\mathfrak{D}}{\sum_{N\neq M}}\left|NM\right|^{-\nicefrac{1}{2}}\\
 & = & \sum_{\deg N<\mathfrak{D}}\left|N\right|^{-1}-\frac{1}{\phi\left(Q\right)}\sum_{\deg N,\:\deg M<\mathfrak{D}}\left|NM\right|^{-\nicefrac{1}{2}}\\
 & = & \sum_{n=0}^{\mathfrak{D}-1}q^{-n}q^{n}-\frac{1}{\phi(Q)}\left(\sum_{n=0}^{\mathfrak{D}-1}q^{-\nicefrac{n}{2}}q^{n}\right)^{2}\\
 & = & \mathfrak{D}-\frac{1}{\phi\left(Q\right)}\left(\frac{q^{\nicefrac{\mathfrak{D}}{2}}-1}{q^{\nicefrac{1}{2}}-1}\right)^{2}
\end{eqnarray*}
and since $\phi\left(Q\right)=\left|Q\right|-1=q^{\mathfrak{D}}-1$,
we get
\[
=\mathfrak{D}-\frac{1}{q^{\mathfrak{\nicefrac{D}{2}}}+1}\frac{q^{\nicefrac{\mathfrak{D}}{2}}-1}{\left(q^{\nicefrac{1}{2}}-1\right)^{2}}=\mathfrak{D}+O\left(1\right).
\]
This completes the proof. \end{proof}
\begin{cor}
Let $Q\in\mathcal{A}$ be an irreducible polynomial, then 
\[
\left|X\left(Q\right)\right|\gg\frac{q^{\mathfrak{D}}}{\mathfrak{D}}.
\]

Since $\left|\left\{ \chi\neq\chi_{0}\mbox{ mod }Q\right\} \right|=q^{\mathfrak{D}}-1$,
we get
\begin{equation}
\frac{\left|X\left(Q\right)\right|}{\left|\left\{ \chi\neq\chi_{0}\mbox{ mod }Q\right\} \right|}\gg\frac{1}{\mathfrak{D}}.\label{eq:coro, sec&first}
\end{equation}
\end{cor}
\begin{proof}
By the Cauchy-Schwarz inequality we have 
\[
\sum_{\chi\in X\left(Q\right)}\left|L\left(\frac{1}{2},\chi\right)\right|\leq\sqrt{\sum_{\chi\in X\left(Q\right)}1^{2}}\sqrt{\sum_{\chi\in X\left(Q\right)}\left|L\left(\frac{1}{2},\chi\right)\right|^{2}},
\]
 hence
\[
\left|X\left(Q\right)\right|\geq\frac{\left(\sum_{\chi\neq\chi_{0}}\left|L\left(\frac{1}{2},\chi\right)\right|\right)^{2}}{\sum_{\chi\neq\chi_{0}}\left|L\left(\frac{1}{2},\chi\right)\right|^{2}},
\]
and by Propositions \propref{first momen}\propref{second moment},
we arrive at 
\[
\left|X\left(Q\right)\right|\gg\frac{\left(q^{\mathfrak{D}}\right)^{2}}{q^{\mathfrak{D}}\mathfrak{D}}=\frac{q^{\mathfrak{D}}}{\mathfrak{D}}.
\]

\end{proof}

\section{\label{sec:The-Selberg-Sieve}The Selberg Sieve for Polynomials}

Sieve methods are of general techniques designed to estimate the size
of sifted sets of integers. The Eratosthenes sieve was the first method
to estimate this size, and around 1946 Atle Selberg introduced a new
method for finding upper bounds to the sieve estimate, \cite{key-23}.
This method gives much better bounds than Eratosthenes sieve. The
polynomial analog to Selberg sieve was proved by Webb \cite{key-25}.
In this section we present the polynomial analog to the Selberg sieve
and prove a special case, similar to the one found in \cite{key-24}.
We will use this case in section \ref{sec:Divisor-Sums}, where we
will prove a Theorem we need for one of our main results, the asymptotic
formula of the fourth moment of Dirichlet L-functions in the polynomial
ring. 

Let $x\geq1$, and let $a_{N}$ be a non-negative real number for
each $N$ monic polynomial, which holds 

\[
\sum_{\deg N\leq x}a_{N}=A<\infty.
\]
For each square free polynomial $D$ let 
\[
A_{D}:=\sum_{\deg M\leq x-\deg D}a_{MD}=A\alpha\left(D\right)+r\left(D\right),
\]
where $\alpha$ is a multiplicative function with $0\leq\alpha(D)\leq1$
for each $D$. Also define

\[
A\left(\mathscr{D}\right):=\sum_{(N,\mathcal{D})=1}a_{N}.
\]

The Selberg sieve: 
\begin{thm}
\label{thm:selberg sieve}For each $z\geq1$,and $\mathscr{D}$ a
square free polynomial, we have 

\[
A\left(\mathscr{D}\right)\leq\frac{A}{S(\mathscr{D},z)}+R\left(\mathscr{D},z\right)
\]

where $S,R$ are defined by 

\[
S\left(\mathscr{D},z\right)=\underset{\deg D\leq z}{\underset{D\mid\mathscr{D}}{\sum}}\prod_{P\mid D}\frac{\alpha\left(P\right)}{1-\alpha\left(P\right)},\qquad R\left(\mathscr{D},z\right)=\underset{\deg D\leq2z}{\underset{D\mid\mathscr{D}}{\sum}}3^{\omega\left(D\right)}\left|r\left(D\right)\right|
\]

and $\omega\left(D\right)=\sum_{P\mid d}1$.
\end{thm}

\subsection{A Special Case }

In our examination in section \ref{sec:Divisor-Sums} we have a particular
case in which we know the following
\[
\alpha\left(D\right)=\frac{1}{\left|D\right|},\; P_{z}=\prod_{\deg P\leq z}P,\; r\left(D\right)=O\left(1\right).
\]
By Selberg sieve, according to the conditions above, we find 
\[
A\left(P_{z}\right)\leq\frac{A}{z}+O\left(q^{2z}\right).
\]

\begin{prop}
For $x\geq1$, $K$ polynomial, we can define
\[
H_{K}(x):=\underset{(D,K)=1}{\sum_{\deg D<x}}\frac{\mu^{2}\left(D\right)}{\phi\left(D\right)}
\]
 and then 
\[
H_{K}(x)\geq\prod_{P\mid K}\left(1-\frac{1}{\left|P\right|}\right)x.
\]
\end{prop}
\begin{proof}
For $\deg K\geq1$ we have 
\begin{eqnarray*}
H_{1}\left(x\right) & = & \sum_{\deg D<x}\frac{\mu^{2}\left(D\right)}{\phi\left(D\right)}=\sum_{L\mid K}\underset{(D,K)=L}{\sum_{\deg D<x}}\frac{\mu^{2}\left(D\right)}{\phi\left(D\right)}\\
 & = & \sum_{L\mid K}\underset{(H,\frac{K}{L})=(H,L)=1}{\sum_{\deg H<x-\deg L}}\frac{\mu^{2}\left(LH\right)}{\phi\left(LH\right)}\\
 & = & \sum_{L\mid K}\frac{\mu^{2}\left(L\right)}{\phi\left(L\right)}\underset{(H,K)=1}{\sum_{\deg H<x-\deg L}}\frac{\mu^{2}\left(H\right)}{\phi\left(H\right)}\\
 & = & \sum_{L\mid K}\frac{\mu^{2}\left(L\right)}{\phi\left(L\right)}H_{K}\left(x-\deg L\right)\\
 & \leq & \left(\sum_{L\mid K}\frac{\mu^{2}\left(L\right)}{\phi\left(L\right)}\right)H_{K}\left(x\right).
\end{eqnarray*}
 Now, since $\frac{\mu^{2}}{\phi}$ is a multiplicative function we
know 
\[
\sum_{L\mid K}\frac{\mu^{2}\left(L\right)}{\phi\left(L\right)}=\prod_{P\mid K}\left(1+\frac{1}{\left|P\right|-1}\right)=\prod_{P\mid K}\left(1-\frac{1}{\left|P\right|}\right)^{-1}=\frac{\left|K\right|}{\phi\left(K\right)}.
\]
 We also have 
\begin{eqnarray*}
H_{1}(x) & = & \sum_{\deg D<x}\frac{\mu^{2}\left(D\right)}{\left|D\right|}\prod_{P\mid D}\left(1-\frac{1}{\left|P\right|}\right)^{-1}\\
 & = & \sum_{\deg D<x}\frac{\mu^{2}\left(D\right)}{\left|D\right|}\prod_{P\mid D}\left(1+\frac{1}{\left|P\right|}+\frac{1}{\left|P\right|^{2}}+\ldots\right),
\end{eqnarray*}
 if we define $\kappa(N)$ to be the maximal square free divisor of
$N$, we get 
\[
=\sum_{\deg\kappa(N)<x}\frac{1}{\left|N\right|}\geq\sum_{\deg N<x}\frac{1}{\left|N\right|}=x.
\]
\end{proof}
\begin{cor}
\label{cor:selberg special case }For $\alpha\left(D\right)=\frac{1}{\left|D\right|}$,
$P_{z}=\prod_{\deg P\leq z}P$ and $z\geq1$, we have
\[
S\left(P_{z},z\right)\geq z.
\]
\end{cor}
\begin{proof}
For $y\geq1$ 
\begin{eqnarray*}
S\left(P_{z},y\right) & = & \underset{\deg D\leq y}{\underset{D\mid P_{z}}{\sum}}\prod_{P\mid D}\frac{\nicefrac{1}{\left|P\right|}}{1-\nicefrac{1}{\left|P\right|}}=\underset{\deg D\leq y}{\underset{D\mid P_{z}}{\sum}}\frac{1}{\phi\left(D\right)}\\
 & = & \underset{\deg D\leq y}{\underset{\left(D,P_{z}^{c}\right)=1}{\sum}}\frac{\mu^{2}\left(D\right)}{\phi\left(D\right)}\geq\prod_{\deg P\leq z}\left(1-\frac{1}{\left|P\right|}\right)\left(y+1\right)\\
 & \geq & \prod_{\deg P\leq z}\left(1-\frac{1}{\left|P\right|}\right)y
\end{eqnarray*}
when $P_{z}^{c}=\frac{\prod_{\deg P\leq y}\left|P\right|}{P_{z}}=\prod_{z<\deg P\leq y}\left|P\right|$.
If we take $y=z$ we get the Corollary.\end{proof}
\begin{prop}
\label{prop:selberg error}If $r\left(D\right)=O\left(1\right)$ for
each square free polynomial d, then
\[
R\left(D,z\right)=O\left(q^{2z}\right).
\]
\end{prop}
\begin{proof}
By definition 
\begin{eqnarray*}
R(\mathscr{D},y) & = & \underset{D_{1},D_{2}\mid\mathscr{D}}{\sum\sum}r\left(\left[D_{1},D_{2}\right]\right)\left|\lambda_{D_{1}}\lambda_{D_{2}}\right|\leq\underset{\deg D_{1},\deg D_{2}\leq z}{\underset{D_{1},D_{2}\mid\mathscr{D}}{\sum}}r\left(\left[D_{1},D_{2}\right]\right)\\
 & \ll & \underset{\deg D_{1},\deg D_{2}\leq z}{\underset{D_{1},D_{2}\mid\mathscr{D}}{\sum}}1\leq\sum_{\deg D_{1}\leq z}\sum_{\deg D_{2}\leq z}1=\left(q^{z+1}\right)^{2}.
\end{eqnarray*}

\end{proof}

\section{\label{sec:Divisor-Sums}Divisor Sums}

In 1980 P. Shiu \cite{key-3} proved that for arithmetic functions
$f$ which satisfy similar properties as the divisor function for
integers, $0<\alpha<\frac{1}{2},\:0<\beta<\frac{1}{2}$ and $a,\: k$
are integers satisfying $0<a<k,\;\left(a,k\right)=1$, then, as $x\rightarrow\infty$, 

\[
\underset{n\equiv a\pmod k}{\underset{x-y<n\leq x}{\sum}}f\left(n\right)\ll\frac{y}{\phi\left(k\right)}\frac{1}{\log x}\exp\left(\underset{p\nmid k}{\underset{p\leq x}{\sum}}\frac{f\left(p\right)}{p}\right),
\]
uniformly in $a,\: k$ and $y$ provided that $k<y^{1-\alpha},\; x^{\beta}<y\leq x$. 

Here we give a polynomial analog to a special case of the above, to
be used it in section \ref{sec:The-Fourth-Moment}. 
\begin{thm}
\label{thm:divisor sum}Let $0<\alpha<1$, and let $A,\: K$ be monic
polynomials, $K$ prime, satisfying $\left(A,K\right)=1$. Then, as
$x\rightarrow\infty$

\[
\underset{N\equiv A\pmod K}{\underset{\deg N\leq x}{\sum}}d\left(N\right)\ll\frac{q^{x}x}{\phi\left(K\right)}
\]
uniformly in $A$ and $K$ provided that $\deg K<\left(1-\alpha\right)x$.
\end{thm}
Note that since it is known that 
\[
\sum_{\deg N\leq x}d\left(N\right)\ll q^{x}x,\quad\frac{\#\left\{ N\::\:\deg N\leq x,\: N\equiv A\pmod K\right\} }{\#\left\{ N\::\:\deg N\leq x\right\} }\approx\frac{1}{\phi\left(K\right)},
\]
we expect to find an approximation of the above form. 

To prove this Theorem we shall require the following Lemmas:

\subsection{Preliminary Lemmas}
\begin{lem}
\label{lem: psi}For all sufficiently large x 

\[
\Psi\left(x,\log_{q}x+\log_{q}\log_{q}x\right)\leq q^{\frac{3x}{\left(\log_{q}x\right)^{\nicefrac{1}{2}}}}.
\]
\end{lem}
\begin{proof}
Let $0\leq\delta\leq1$. For all large $y$ we have 
\begin{eqnarray*}
\underset{\deg P\leq y}{\sum_{P\mbox{ prime}}}\frac{1}{\deg P} & = & \underset{\deg P\leq\delta y}{\sum_{P\mbox{ prime}}}\frac{1}{\deg P}+\underset{\delta y\leq\deg P\leq y}{\sum_{P\mbox{ prime}}}\frac{1}{\deg P}\\
 & \leq & q^{\delta y}+\left(\delta y\right)^{-2}\underset{\delta y\leq\deg P\leq y}{\sum_{P\mbox{ prime}}}\deg p\\
 & = & q^{\delta y}+\left(\delta y\right)^{-2}\sum_{\delta y\leq n\leq y}n\pi\left(n\right),
\end{eqnarray*}
 by prime polynomials theorem we know $n\pi\left(n\right)\leq q^{n}$
(as discussed in subsection \ref{sub:background.1}), so we arrive
at 
\begin{eqnarray*}
 & \leq & q^{\delta y}+\left(\delta y\right)^{-2}\sum_{\delta y\leq n\leq y}q^{n}\\
 & = & q^{\delta y}+\left(\delta y\right)^{-2}\frac{q^{y}-q^{\left\lceil \delta y\right\rceil -1}}{1-\nicefrac{1}{q}}\\
 & \leq & q^{\delta y}+\left(2+\epsilon\right)\frac{q^{y}}{\left(\delta y\right)^{2}}
\end{eqnarray*}
for $\epsilon>0$. By setting $\delta=\sqrt{\nicefrac{25}{26}},\:\epsilon=\nicefrac{1}{12}$,
we get 
\[
\underset{\deg P\leq y}{\sum_{P\mbox{ prime}}}\frac{1}{\deg P}\leq2\frac{1}{6}\frac{q^{y}}{y^{2}}.
\]

We now use Rankin's method. Let $\delta>0$. We have, for large $y$,

\begin{eqnarray*}
\Psi\left(x,y\right) & = & \underset{p_{+}\left(N\right)\leq y}{\underset{\deg N\leq x}{\sum}}1\leq q^{\delta x}\underset{p_{+}\left(N\right)\leq y}{\underset{\deg N\leq x}{\sum}}\frac{1}{\left|N\right|^{\delta}}\\
 & \leq & q^{\delta x}\underset{p_{+}\left(N\right)\leq y}{\underset{N\:\mbox{monic}}{\sum}}\frac{1}{\left|N\right|^{\delta}}\\
 & = & q^{\delta x}\underset{\deg P\leq y}{\prod}\left(1+\frac{1}{\left|P\right|^{\delta}}+\frac{1}{\left|P\right|^{2\delta}}+\ldots\right)\\
 & = & q^{\delta x}\underset{\deg P\leq y}{\prod}\left(1+\frac{1}{\left|P\right|^{\delta}-1}\right)\\
 & \leq & q^{\delta x}\exp\left(\underset{\deg P\leq y}{\sum}\frac{1}{\left|P\right|^{\delta}}\right)\leq q^{\delta x}\exp\left(\frac{1}{\delta}\underset{\deg P\leq y}{\sum}\frac{1}{\deg P}\right)\\
 & \leq & q^{\delta x}\exp\left(\frac{13}{6}\frac{q^{y}}{\delta y^{2}}\right)\leq q^{\delta x+\nicefrac{13}{6}\frac{q^{y}}{\delta y^{2}}}.
\end{eqnarray*}
When $y=\log_{q}x+\log_{q}\log_{q}x$ and set $\delta=\nicefrac{3}{2}\log_{q}^{-\nicefrac{1}{2}}x$,
the result follows.\end{proof}
\begin{lem}
\label{lem: phi}Let $A$ and $K$ satisfy $\deg A<\deg K,\;\left(A,K\right)=1$.
Assume that $K$ is a prime polynomial which satisfies $\deg K<x$
and $z\geq2$. Then 

\[
\Phi\left(x,z;K,A\right)\leq\frac{q^{x}}{\phi\left(K\right)z}+q^{2z}.
\]
\end{lem}
\begin{proof}
We use the Selberg sieve method (\secref{The-Selberg-Sieve}) to prove
this Lemma. We know 
\[
B:=\underset{N\equiv A\pmod K}{\sum_{\deg N<x}}1=\frac{q^{x}}{\phi\left(K\right)}<\infty,
\]
also, for each polynomial $D$ 
\[
A_{D}:=\underset{N\equiv0\pmod D}{\underset{N\equiv A\pmod K}{\underset{\deg N<x}{\sum}}}1=B\alpha\left(D\right)+r\left(D\right),
\]
when $\alpha\left(D\right)=\frac{1}{\left|D\right|}$ and $r\left(D\right)=O\left(1\right)$.
Since $\deg N\leq x$ and $K$ prime we know $\left(N,K\right)=1$
and can estimate 
\[
\#\{N:\:\deg N\leq x,\: N\equiv A\pmod K,\: D\mid N\}=\frac{q^{x}}{\phi\left(k\right)\left|D\right|}+O\left(1\right),
\]
$\alpha$ is a multiplicative function with $0\leq\alpha\left(D\right)\leq1$
for each $D$. Let $P_{z}:=\underset{\deg P\leq z}{\prod}P$, from
the polynomial analog to the Selberg sieve, we know

\[
\Phi\left(x,z;K,A\right)=\underset{p_{-}\left(N\right)>z}{\underset{N\equiv A\pmod K}{\underset{\deg N<x}{\sum}}}1=\underset{\left(N,P_{z}\right)=1}{\underset{N\equiv A\pmod K}{\underset{\deg N<x}{\sum}}}1\leq\frac{B}{S\left(P_{z},z\right)}+R\left(P_{z},z\right).
\]
By the special case, Corollary \ref{cor:selberg special case }, we
get 

\begin{eqnarray*}
S\left(P_{z},z\right) & = & \underset{\deg D\leq z}{\underset{D\mid P_{z}}{\sum}}\prod_{P\mid D}\frac{1}{\left|P\right|-1}\geq z
\end{eqnarray*}
since for all $\deg D\leq z$, $D\mid P_{z}$, and by Proposition
\ref{prop:selberg error} we have 

\[
R\left(P_{z},z\right)=\underset{\deg D\leq2z}{\underset{D\mid P_{z}}{\sum}}3^{\omega(D)}\left|r\left(D\right)\right|\ll q^{2z}.
\]
\end{proof}
\begin{lem}
\label{lem: divisor sum}For all integers $x>0$

\[
\sum_{\deg N\leq x}\frac{d\left(N\right)}{\left|N\right|}=\frac{\left(x+1\right)\left(x+2\right)}{2}.
\]
\end{lem}
\begin{proof}
We have

\begin{eqnarray*}
\sum_{\deg N\leq x}\frac{d\left(N\right)}{\left|N\right|} & = & \sum_{\deg A\leq x}\sum_{\deg B\leq x-\deg A}\frac{1}{\left|AB\right|}\\
 & = & \sum_{n=0}^{x}\sum_{\deg A=n}\frac{1}{\left|A\right|}\sum_{m=0}^{x-n}\sum_{\deg B=m}\frac{1}{\left|B\right|}\\
 & = & \sum_{n=0}^{x}\sum_{m=0}^{x-n}1=\sum_{n=0}^{x}x-n+1\\
 & = & \frac{\left(x+1\right)\left(x+2\right)}{2}.
\end{eqnarray*}
\end{proof}
\begin{lem}
\label{lem: div sum bound}As $z\rightarrow\infty$

\[
\underset{p_{+}(N)\leq\frac{z}{r}}{\underset{\deg N\geq\frac{z}{2}}{\sum}}\frac{d(N)}{\left|N\right|}\ll z^{2}q^{-\frac{1}{10}r\log_{q}r}
\]
provided that $1\leq r\leq\frac{z}{\log_{q}z}$.\end{lem}
\begin{proof}
Again we use Rankin's method. Let $\frac{3}{4}\leq\delta\leq1$. We
have 

\begin{eqnarray*}
\underset{p_{+}\left(N\right)\leq y}{\underset{\deg N\geq x}{\sum}}\frac{d\left(N\right)}{\left|N\right|} & \leq & q^{\left(\delta-1\right)x}\underset{p_{+}\left(N\right)\leq y}{\underset{\deg N\geq x}{\sum}}\frac{d\left(N\right)}{\left|N\right|^{\delta}}\leq q^{\left(\delta-1\right)x}\underset{p_{+}\left(N\right)\leq y}{\underset{N\:\mbox{monic}}{\sum}}\frac{d\left(N\right)}{\left|N\right|^{\delta}}\\
 & = & q^{\left(\delta-1\right)x}\prod_{\deg P\leq y}\left(\sum_{l=0}^{\infty}\frac{d\left(P^{l}\right)}{\left|P\right|^{\delta l}}\right)\ll q^{\left(\delta-1\right)x}\exp\left(2\sum_{\deg P\leq y}\frac{1}{\left|P\right|^{\delta}}\right).
\end{eqnarray*}
Now, $\frac{1}{\left|P\right|^{\delta}}=\frac{1}{\left|P\right|}+\frac{1}{\left|P\right|}(\left|P\right|^{1-\delta}-1)$.
By the Taylor series for the exponential function 

\begin{eqnarray*}
\sum_{\deg P\leq y}\frac{1}{\left|P\right|}\left(\left|P\right|^{1-\delta}-1\right) & \leq & \sum_{\deg P\leq y}\frac{1}{\left|P\right|}\sum_{n=1}^{\infty}\frac{\left(\left(1-\delta\right)\ln\left|P\right|\right)^{n}}{n!}\\
 & \leq & \sum_{n=1}^{\infty}\frac{\left(1-\delta\right)^{n}y^{n-1}\ln^{n}q}{n!}\sum_{\deg P\leq y}\frac{\deg P}{\left|P\right|}\\
 & \leq & 2\sum_{n=1}^{\infty}\frac{\left(\left(1-\delta\right)y\ln q\right)^{n}}{n!}\\
 & \leq & 2\exp\left(\left(1-\delta\right)y\ln q\right)=2q^{\left(1-\delta\right)y}
\end{eqnarray*}
when the third inequality is due to PNT, since $n\pi\left(n\right)\leq q^{n}$
we get
\[
\sum_{\deg P\leq y}\frac{\deg P}{\left|P\right|}=\sum_{n\leq y}\frac{n}{q^{n}}\pi\left(n\right)\leq\sum_{n\leq y}1=y+1\leq2y.
\]
Therefore, we have 

\begin{eqnarray*}
\underset{p_{+}\left(N\right)\leq y}{\underset{\deg N\geq x}{\sum}}\frac{d(N)}{\left|N\right|} & \ll & q^{\left(\delta-1\right)x}\exp\left(2\sum_{\deg P\leq y}\frac{1}{\left|P\right|}+2q^{\left(1-\delta\right)y}\right)\\
 & \ll & q^{\left(\delta-1\right)x}\exp\left(2\log y+4q^{\left(1-\delta\right)y}\right).
\end{eqnarray*}
Now put $x=\frac{z}{2},\: y=\frac{z}{r}$ and set $\delta=1-\frac{r\log_{q}r}{4z}$.
Note that if $1\leq r\leq\frac{z}{\log z}$ then $r\log_{q}r<z$ and
so $\frac{3}{4}<\delta\leq1$.

\[
\underset{p_{+}\left(N\right)\leq\frac{z}{r}}{\underset{\deg N\geq\frac{z}{2}}{\sum}}\frac{d\left(N\right)}{\left|N\right|}\ll q^{-\frac{r\log_{q}r}{4z}\frac{z}{2}}z^{2}\exp\left(4q^{\frac{r\log_{q}r}{4z}\frac{z}{r}}\right)=z^{2}q^{-\frac{1}{8}r\log_{q}r}e^{4r^{\frac{1}{4}}}
\]
and the result follows.
\end{proof}

\subsection{Proof of Theorem \ref{thm 4th}}

Let $K$ as required above, and put $z=\frac{\alpha}{10}x$.

For each $N$ satisfying $\deg N\leq x,\: N\equiv A\pmod K$ we express
$N$ in the form 

\[
N=P_{1}^{s_{1}}\cdots P_{j}^{s_{j}}P_{j+1}^{s_{j+1}}\cdots P_{l}^{s_{l}}=B_{N}D_{N},\qquad\left(\deg P_{1}\leq\deg P_{2}\leq\cdots\leq\deg P_{l}\right)
\]
where $B_{n}$ is chosen so that 

\[
\deg B_{N}\leq z<\deg B_{N}P_{j+1}^{s_{j+1}}.
\]
We divide the set of such polynomials into the following classes:

\begin{eqnarray*}
\mathrm{I} &  & p_{-}\left(D_{N}\right)>\frac{z}{2}\\
\mathrm{II} &  & p_{-}\left(D_{N}\right)\leq\frac{z}{2},\;\deg B_{N}\leq\frac{z}{2}\\
\mathrm{III} &  & p_{-}(\left(D_{N}\right)\leq\log_{q}z+\log_{q}\log_{q}z,\;\deg B_{N}>\frac{z}{2}\\
\mathrm{IV} &  & \log_{q}z+\log_{q}\log_{q}z<p_{-}\left(D_{N}\right)\leq\frac{z}{2},\;\deg B_{N}>\frac{z}{2}.
\end{eqnarray*}

First we have 

\begin{eqnarray*}
\sum_{N\in\mathrm{I}}d\left(N\right) & = & \sum_{N\in\mathrm{I}}d\left(B_{N}\right)d\left(D_{N}\right)\leq\sum_{\deg B\leq z}d\left(B\right)\underset{p_{-}\left(\frac{N}{B}\right)>\frac{z}{2}}{\underset{N\equiv0\pmod B}{\underset{N=A\pmod K}{\underset{\deg N\leq x}{\sum}}}}d\left(\frac{N}{B}\right)\\
 & = & \sum_{\deg B\leq z}d\left(B\right)\underset{p_{-}\left(D\right)>\frac{z}{2}}{\underset{D\equiv A'\pmod K}{\underset{\deg D\leq x-\deg B}{\sum}}}d\left(D\right)
\end{eqnarray*}
where $A'\equiv A\bar{B},\; B\bar{B}\equiv1\pmod K$. By definition
we know $p_{-}\left(D\right)>\frac{z}{2}=\frac{\alpha}{20}x$, so
that $\Omega\left(D\right)\leq\frac{x}{p_{-}\left(D\right)}<\frac{20}{\alpha}$,
therefore, $d\left(D\right)=2^{\Omega\left(D\right)}\leq2^{\frac{20}{\alpha}}$.
Hence, we have

\[
\sum_{N\in I}d\left(N\right)\ll\underset{\left(B,K\right)=1}{\underset{\deg B\leq z}{\sum}}d\left(B\right)\underset{p_{-}\left(D\right)>\frac{z}{2}}{\underset{D\equiv A'\pmod K}{\underset{\deg D\leq x-\deg B}{\sum}}}1.
\]
Since $\left|K\right|<q^{\left(1-\alpha\right)x}$ and $\left|B\right|\leq q^{z}<q^{\alpha x}$,
so that $\left|KB\right|<x$, it follows from Lemma \ref{lem: phi}
that $\Phi\left(x-\deg B,\frac{z}{2};K,A'\right)\leq\frac{2q^{x}}{\phi\left(K\right)\left|B\right|z}+q^{z}$,
and accordingly

\begin{eqnarray*}
\sum_{N\in\mathrm{I}}d\left(N\right) & \ll & \left(\frac{2q^{x}}{\phi\left(K\right)z}+q^{2z}\right)\underset{\left(B,K\right)=1}{\underset{\deg B\leq z}{\sum}}\frac{d\left(B\right)}{\left|B\right|}.
\end{eqnarray*}
From Lemma \ref{lem: divisor sum} we arrive at

\begin{equation}
\sum_{N\in\mathrm{I}}d\left(N\right)\ll\left(\frac{2q^{x}}{\phi\left(K\right)z}+q^{2z}\right)z^{2}.\label{eq:div I}
\end{equation}

Next, to each $N\in\mathrm{II}$, there are corresponding $P$ and
$s$ such that $P^{s}\parallel N,$\linebreak{}
 $\deg P\leq\frac{z}{2}$ and $\deg P^{s}>\frac{z}{2}$ (for example,
$P_{j+1}$ will hold). Let $s_{P}$ denote the least positive integer
$s$ satisfying $\deg P^{s}>\frac{z}{2}$ so that $s_{P}\geq2$ and
hence\linebreak{}
 $\left|P\right|^{-s_{P}}\leq\min\left(q^{-\frac{z}{2}},\left|P\right|^{-2}\right)$;
Thus 
\[
\sum_{\deg P\leq\frac{z}{2}}\frac{1}{\left|P\right|^{s_{P}}}\leq\sum_{\deg P\leq\frac{z}{4}}q^{-\frac{z}{2}}+\sum_{\deg P>\frac{z}{4}}\left|P\right|^{-2}\ll q^{-\frac{z}{4}}.
\]
It now follows that 

\[
\sum_{N\in\mathrm{II}}1\leq\sum_{\deg P\leq\frac{z}{2}}\underset{N\equiv0\pmod P^{s_{P}}}{\underset{N\equiv A\pmod K}{\underset{\deg N\leq x}{\sum}}}1=\underset{P\nmid K}{\underset{\deg P\leq\frac{z}{2}}{\sum}}\left(\frac{q^{x}}{\left|K\right|\left|P^{s_{P}}\right|}+O\left(1\right)\right)\ll\frac{q^{x}}{\left|K\right|}q^{-\frac{z}{4}}+q^{\frac{z}{2}}.
\]

Suppose next that $N\in\mathrm{III}$. $B$ exists such that $B\mid N,\;\frac{z}{2}<\deg B\leq z$,
and $p_{+}(B)<\log_{q}z+\log_{q}\log_{q}z$ (for example, $B_{N}$).
Consequently we have

\begin{eqnarray*}
\sum_{N\in\mathrm{III}}1 & \leq & \underset{p_{+}\left(B\right)<\log_{q}x+\log_{q}\log_{q}x}{\underset{\frac{z}{2}<\deg B\leq z}{\sum}}\underset{N\equiv0\pmod B}{\underset{N\equiv A\pmod K}{\underset{\deg N\leq x}{\sum}}}1\\
 & = & \underset{\left(B,K\right)=1}{\underset{p_{+}\left(B\right)<\log_{q}z+\log_{q}\log_{q}z}{\underset{\frac{z}{2}<\deg B\leq z}{\sum}}}\left(\frac{q^{x}}{\left|K\right|\left|B\right|}+O\left(1\right)\right)\\
 & \leq & \frac{q^{x}}{\left|K\right|}q^{-\frac{z}{2}}\Psi\left(z,\log_{q}z+\log_{q}\log_{q}z\right)+O\left(q^{z}\right)\\
 & \ll & \frac{q^{x}}{\left|K\right|}q^{-\frac{z}{4}}+q^{z}
\end{eqnarray*}
by Lemma \ref{lem: psi}. Since $\deg K<(1-\alpha)x$ we have, by
the definition of z, 
\[
q^{z}<q^{\alpha x}q^{-\frac{z}{4}}<\frac{q^{x}}{\left|K\right|}q^{-\frac{z}{4}}.
\]
 And so, out of the last two inequality

\[
\sum_{N\in\mathrm{II}}1+\sum_{N\in\mathrm{III}}1\ll\frac{q^{x}}{\left|K\right|}q^{-\frac{z}{4}}.
\]
We have (by the properties of the divisor function and the definition
of z) \linebreak{}
$d\left(N\right)\ll\left|N\right|^{\frac{\alpha}{80}}\leq q^{\frac{\alpha}{80}x}=q^{\frac{z}{8}}$.
Consequently, 

\begin{equation}
\sum_{N\in\mathrm{II}}d\left(N\right)+\sum_{N\in\mathrm{III}}d\left(N\right)\ll\frac{q^{x}}{\left|K\right|}q^{-\frac{z}{8}}.\label{eq:div II+III}
\end{equation}

Lastly we address the class $\mathrm{IV}$. We have 

\[
\sum_{N\in\mathrm{IV}}d\left(N\right)=\sum_{N\in\mathrm{IV}}d\left(B_{N}\right)d\left(D_{N}\right)\leq\sum_{\frac{z}{2}<\deg B\leq z}d\left(B\right)\underset{\log_{q}z+\log_{q}\log_{q}z<p_{-}(D_{N})\leq\frac{z}{2}}{\underset{B_{N}=B,\: p_{-}(D_{N})>p_{+}(B)}{\underset{N\equiv A\pmod K}{\underset{\deg N\leq x}{\sum}}}}d\left(D_{N}\right).
\]
Let us denote $r_{0}:=\left[\frac{z}{\log_{q}z+\log_{q}\log_{q}z}\right]$,
so that $\log_{q}z+\log_{q}\log_{q}z>\frac{z}{r_{0}+1}$. Let $2<r<r_{0}$
\linebreak{}
 and consider these $N$ for which $\frac{z}{r+1}<p_{-}\left(D_{N}\right)\leq\frac{z}{r}$.
For such $N$, we have $p_{+}\left(B_{N}\right)=p_{+}\left(B\right)<p_{-}\left(D_{N}\right)<\frac{z}{r}$
and moreover, as before,

\[
\Omega\left(D_{N}\right)\leq\frac{x}{p_{-}\left(D_{N}\right)}\leq\frac{\left(r+1\right)x}{z}=\frac{10\left(r+1\right)}{\alpha}<\frac{20r}{\alpha}
\]
so that $d\left(D_{N}\right)\leq2^{\Omega\left(D_{N}\right)}\leq\gamma^{r}$
where $\gamma=2^{\frac{20}{\alpha}}$. It follows that 

\begin{eqnarray*}
\sum_{N\in\mathrm{IV}}d\left(N\right) & \leq & \underset{2\leq r\leq r_{0}}{\sum}\gamma^{r}\sum_{\begin{array}{c}
\frac{z}{2}<\deg B\leq z\\
p_{+}\left(B\right)<\frac{z}{r}
\end{array}}d(B)\underset{\frac{z}{r+1}<p_{-}\left(\frac{N}{B}\right)<\frac{z}{r}}{\underset{N\equiv0\pmod B}{\underset{N\equiv A\pmod K}{\underset{\deg N\leq x}{\sum}}}}1\\
 & \leq & \underset{2\leq r\leq r_{0}}{\sum}\gamma^{r}\underset{\left(B,K\right)=1}{\underset{p_{+}\left(B\right)<\frac{z}{r}}{\underset{\frac{z}{2}<\deg B\leq z}{\sum}}}d\left(B\right)\Phi\left(x-\deg B,\frac{z}{r+1};K,A'\right)
\end{eqnarray*}
where $A'\equiv A\bar{B},\; B\bar{B}\equiv1\pmod K$. By Lemma \ref{lem: phi}
we have 

\begin{eqnarray*}
\Phi\left(x-\deg B,\frac{z}{r+1};K,A'\right) & \leq & \frac{q^{x}\left(r+1\right)}{\phi\left(K\right)\left|B\right|z}+q^{\frac{2z}{r+1}}
\end{eqnarray*}
and therefore 

\[
\sum_{N\in\mathrm{IV}}d\left(N\right)\leq\left(\frac{q^{x}}{\phi\left(K\right)z}+q^{2z}\right)\sum_{2\leq r\leq r_{0}}\left(r+1\right)\gamma^{r}\underset{(B,K)=1}{\underset{p_{+}(B)<\frac{z}{r}}{\underset{\frac{z}{2}<\deg B\leq z}{\sum}}}\frac{d\left(B\right)}{\left|B\right|}.
\]
By the definition of $r_{0}$ we see that we can apply Lemma \ref{lem: div sum bound}
to the inner sum, giving

\begin{equation}
\begin{aligned}\sum_{N\in\mathrm{IV}}d\left(N\right) & \ll\left(\frac{q^{x}}{\phi\left(K\right)z}+q^{2z}\right)z^{2}\sum_{2\leq r\leq r_{0}}\left(r+1\right)\gamma^{r}q^{-\frac{1}{10}r\log_{q}r}\\
 & \ll\frac{q^{x}z}{\phi\left(K\right)}+z^{2}q^{2z}.
\end{aligned}
\label{eq:div IV}
\end{equation}

For $\left|K\right|<x^{1-\alpha}$ we have, by the definition of $z$, 

\[
q^{2z}<\frac{\left|K\right|}{\phi\left(K\right)}\frac{q^{3z}}{z}<\frac{q^{\left(1-\alpha\right)x}q^{\frac{3\alpha}{10}x}}{\phi\left(K\right)z}<\frac{q^{x}}{\phi\left(K\right)z}.
\]
In respect to \ref{eq:div I}, \ref{eq:div II+III} and \ref{eq:div IV}
we receive the desired result.

\section{\label{sec:Expressing L as a short sum}Expressing $L\left(s,\chi\right)$
As a Short Sum}

One of our main goal is to find an asymptotic formula to the fourth
moment of the Dirichlet L-function over the polynomial ring. In this
section we reduce this sum into a polynomial of degree $\mathfrak{D}$
to simplify the fourth moment problem. We have

\[
\left|L^{*}\left(u,\chi\right)\right|^{2}=\sum_{\deg N,\deg M\leq\mathfrak{D}}\chi\left(N\right)\bar{\chi}\left(M\right)u^{\deg N+\deg M}=\sum_{n=0}^{2\mathfrak{D}}u^{n}A_{n}
\]
when $A_{n}:=\underset{\deg NM=n}{\sum}\chi\left(N\right)\bar{\chi}\left(M\right)$.
One can consider this reduction as the polynomial analogue of the
approximate functional equation for the Riemann zeta function ,%
\cite{key-12}. Also, there is a similar analogue for the quadratic
Dirichlet L\textendash{}function, \cite{key-10}. Note, that in our
case the formula is an identity rather than an approximation.
\begin{prop}
\label{thm:squared L}For $\mathfrak{D}\geq1$, $\chi$ primitive
character modulo $Q$, we have 

\[
\left|L\left(\frac{1}{2},\chi\right)\right|^{2}=2\sum_{\deg NM<\mathfrak{D}}\frac{\chi\left(N\right)\bar{\chi}\left(M\right)}{\left|NM\right|^{\nicefrac{1}{2}}}+\pi\left(\chi\right)
\]

when

\[
\pi\left(\chi\right):=\begin{cases}
\begin{array}{l}
2q^{-\nicefrac{\mathfrak{D}}{2}}\left(A_{\mathfrak{D}-3}-A_{\mathfrak{D}-2}+A_{\mathfrak{D}-1}\right)\\
\qquad-\left(q^{-\nicefrac{\left(\mathfrak{D}-1\right)}{2}}-q^{-\nicefrac{\left(\mathfrak{D}+1\right)}{2}}\right)A_{\mathfrak{D}-2};
\end{array} & \chi\mbox{ "even"}\\
q^{-\nicefrac{\mathfrak{D}}{2}}A_{\mathfrak{D}}; & \mbox{else.}
\end{cases}
\]
\end{prop}
\begin{proof}
We divide to two cases:

First, if $\chi$ is ``even'' then the functional equation gives

\[
\left|\Lambda\left(u,\chi\right)\right|^{2}=\left(q^{\nicefrac{1}{2}}u\right)^{2\left(\mathfrak{D}-1\right)}\left|\Lambda\left(\frac{1}{qu},\bar{\chi}\right)\right|^{2}.
\]
Also $\left|\Lambda\left(u,\chi\right)\right|^{2}$ is a polynomial
in $u$ of degree $\mathfrak{D}-1$, and we can write \linebreak{}
$\left|\Lambda\left(u,\chi\right)\right|^{2}=\sum_{n=0}^{2\left(\mathfrak{D}-1\right)}u^{n}B_{n}$
(note that if $u$ is real, by properties of complex functions $\left|\Lambda\left(u,\bar{\chi}\right)\right|^{2}=\left|\Lambda\left(u,\chi\right)\right|^{2}$,
it can also be represented this way). Now, by the functional equation

\begin{eqnarray*}
\sum_{m=0}^{2\left(\mathfrak{D}-1\right)}u^{m}B_{m} & = & \left(q^{\nicefrac{1}{2}}u\right)^{2\left(\mathfrak{D}-1\right)}\sum_{n=0}^{2\left(\mathfrak{D}-1\right)}\left(\frac{1}{qu}\right)^{n}B_{n}=\sum_{n=0}^{2\left(\mathfrak{D}-1\right)}q^{\mathfrak{D}-1-n}u^{2\left(\mathfrak{D}-1\right)-n}B_{n}
\end{eqnarray*}
and by writing $m=2\left(\mathfrak{D}-1\right)-n$, we get

\[
=\sum_{m=0}^{2\left(\mathfrak{D}-1\right)}u^{m}q^{m-\mathfrak{D}+1}B_{2\left(\mathfrak{D}-1\right)-m}.
\]
We can conclude 

\begin{equation}
B_{m}=q^{m-\mathfrak{D}+1}B_{2\left(\mathfrak{D}-1\right)-m}\;\Rightarrow\; B_{2\left(\mathfrak{D}-1\right)-m}=q^{\mathfrak{D}-1-m}B_{m}.\label{eq:B_m}
\end{equation}
Now, 

\begin{eqnarray*}
\left|L\left(\frac{1}{2},\chi\right)\right|^{2} & = & \left(1-q^{-\nicefrac{1}{2}}\right)^{2}\sum_{n=0}^{2\left(\mathfrak{D}-1\right)}q^{-\nicefrac{1}{2}n}B_{n}\\
 & = & \left(1-q^{-\nicefrac{1}{2}}\right)^{2}\left[\sum_{n=0}^{\mathfrak{D}-1}q^{-\nicefrac{n}{2}}B_{n}+\sum_{n=\mathfrak{D}}^{2\left(\mathfrak{D}-1\right)}q^{-\nicefrac{n}{2}}B_{n}\right],
\end{eqnarray*}
as shown earlier, $m=2\left(\mathfrak{D}-1\right)-n$ 

\[
=\left(1-q^{-\nicefrac{1}{2}}\right)^{2}\left[\sum_{n=0}^{\mathfrak{D}-1}q^{-\nicefrac{n}{2}}B_{n}+\sum_{m=0}^{\mathfrak{D}-2}q^{-\nicefrac{1}{2}\left[2\left(\mathfrak{D}-1\right)-m\right]}B_{2\left(\mathfrak{D}-1\right)-m}\right],
\]
by equation \ref{eq:B_m} we get

\begin{eqnarray*}
 & = & \left(1-q^{-\nicefrac{1}{2}}\right)^{2}\left[\sum_{n=0}^{\mathfrak{D}-1}q^{-\nicefrac{n}{2}}B_{n}+\sum_{m=0}^{\mathfrak{D}-2}q^{-\nicefrac{m}{2}}B_{m}\right]\\
 & = & \left(1-q^{-\nicefrac{1}{2}}\right)^{2}\left[2\sum_{n=0}^{\mathfrak{D}-2}q^{-\nicefrac{n}{2}}B_{n}+q^{-\nicefrac{\left(\mathfrak{D}-1\right)}{2}}B_{\mathfrak{D}-1}\right]\\
 & = & 2\left[\sum_{n=0}^{\mathfrak{D}-2}q^{-\nicefrac{n}{2}}B_{n}-2\sum_{n=1}^{\mathfrak{D}-1}q^{-\nicefrac{n}{2}}B_{n-1}+\sum_{n=2}^{\mathfrak{D}}q^{-\nicefrac{n}{2}}B_{n-2}\right]\\
 &  & +B_{\mathfrak{D}-1}\left(q^{-\nicefrac{\left(\mathfrak{D}-1\right)}{2}}-2q^{-\nicefrac{\mathfrak{D}}{2}}+q^{-\nicefrac{\left(\mathfrak{D}+1\right)}{2}}\right)\\
 & = & 2[B_{0}+q^{-\nicefrac{1}{2}}\left(B_{1}-2B_{0}\right)+\sum_{n=2}^{\mathfrak{D}-2}q^{-\nicefrac{1}{2}n}\left(B_{n}-2B_{n-1}+B_{n-2}\right)\\
 &  & +q^{-\nicefrac{\left(\mathfrak{D}-1\right)}{2}}\left(-2B_{\mathfrak{D}-2}+B_{\mathfrak{D}-3}\right)+q^{-\nicefrac{\mathfrak{D}}{2}}B_{\mathfrak{D}-2}]\\
 &  & +B_{\mathfrak{D}-1}\left(q^{-\nicefrac{(\mathfrak{D}-1)}{2}}-2q^{-\nicefrac{\mathfrak{D}}{2}}+q^{-\nicefrac{(\mathfrak{D}+1)}{2}}\right).
\end{eqnarray*}
Since $\left|L\left(\frac{1}{2},\chi\right)\right|^{2}=\left(1-q^{-\nicefrac{1}{2}}\right)^{2}\left|\Lambda\left(q^{-\nicefrac{1}{2}},\chi\right)\right|^{2}$, 

\begin{eqnarray*}
\sum_{n=0}^{2\mathfrak{D}}q^{-\nicefrac{n}{2}}A_{n} & = & \left(1-q^{-\nicefrac{1}{2}}\right)^{2}\sum_{n=0}^{2\left(\mathfrak{D}-1\right)}q^{-\nicefrac{n}{2}}B_{n}\\
 & = & B_{0}+q^{-\nicefrac{1}{2}}\left(B_{1}-2B_{0}\right)+\sum_{n=2}^{2\left(\mathfrak{D}-1\right)}q^{-\nicefrac{n}{2}}\left(B_{n}-2B_{n-1}+B_{n-2}\right)\\
 &  & +q^{-\nicefrac{(2\mathfrak{D}-1)}{2}}\left(-2B_{2\mathfrak{D}-2}+B_{2\mathfrak{D}-3}\right)+q^{-\mathfrak{D}}B_{2\mathfrak{D}-2}.
\end{eqnarray*}
We conclude

\[
A_{n}=\begin{cases}
B_{0}, & n=0\\
B_{1}-2B_{0}, & n=1\\
B_{n}-2B_{n-1}+B_{n-2}, & 2\leq n\leq2\mathfrak{D}-2\\
-2B_{2\mathfrak{D}-2}+B_{2\mathfrak{D}-3}, & n=2\mathfrak{D}-1\\
B_{2\mathfrak{D}-2}, & n=2\mathfrak{D}
\end{cases}
\]
and

\[
B_{n}=\begin{cases}
A_{0}, & n=0\\
A_{1}+2A_{0}, & n=1\\
A_{n-1}+\overset{n}{\underset{k=1}{\sum}}A_{k}, & 2\leq n\leq2\mathfrak{D}-2\\
A_{2\mathfrak{D}-1}+2A_{2\mathfrak{D}}, & n=2\mathfrak{D}-1\\
A_{2\mathfrak{D}}, & n=2\mathfrak{D}-2
\end{cases}
\]
which yields

\begin{eqnarray*}
\left|L\left(\frac{1}{2},\chi\right)\right|^{2} & = & 2[A_{0}+q^{-\nicefrac{1}{2}}A_{1}+\sum_{n=2}^{\mathfrak{D}-2}q^{-\nicefrac{n}{2}}A_{n}+q^{-\nicefrac{(\mathfrak{D}-1)}{2}}\left(A_{\mathfrak{D}-1}-B_{\mathfrak{D}-1}\right)\\
 &  & +q^{-\nicefrac{\mathfrak{D}}{2}}B_{\mathfrak{D}-2}]+\left(q^{-\nicefrac{\left(\mathfrak{D}-1\right)}{2}}-2q^{-\nicefrac{\mathfrak{D}}{2}}+q^{-\nicefrac{\left(\mathfrak{D}+1\right)}{2}}\right)B_{\mathfrak{D}-1}\\
 & = & 2\sum_{n=0}^{\mathfrak{D}-1}q^{-\nicefrac{n}{2}}A_{n}+2q^{-\nicefrac{\mathfrak{D}}{2}}B_{\mathfrak{D}-2}-(q^{-\nicefrac{(\mathfrak{D}-1)}{2}}+2q^{-\nicefrac{\mathfrak{D}}{2}}\\
 &  & -q^{-\nicefrac{\left(\mathfrak{D}+1\right)}{2}})B_{\mathfrak{D}-1}\\
 & = & 2\sum_{n=0}^{\mathfrak{D}-1}q^{-\nicefrac{n}{2}}A_{n}+2q^{-\nicefrac{\mathfrak{D}}{2}}\left(A_{\mathfrak{D}-3}-A_{\mathfrak{D}-2}+A_{\mathfrak{D}-1}\right)\\
 &  & -\left(q^{-\nicefrac{\left(\mathfrak{D}-1\right)}{2}}-q^{-\nicefrac{\left(\mathfrak{D}+1\right)}{2}}\right)\left(A_{\mathfrak{D}-2}+\sum_{k=0}^{\mathfrak{D}-1}A_{k}\right)\\
 & = & 2\sum_{n=0}^{\mathfrak{D}-1}q^{-\nicefrac{n}{2}}A_{n}+2q^{-\nicefrac{\mathfrak{D}}{2}}\left(A_{\mathfrak{D}-3}-A_{\mathfrak{D}-2}+A_{\mathfrak{D}-1}\right)\\
 &  & -\left(q^{-\nicefrac{\left(\mathfrak{D}-1\right)}{2}}-q^{-\nicefrac{\left(\mathfrak{D}+1\right)}{2}}\right)A_{\mathfrak{D}-2},
\end{eqnarray*}
when the last equality holds according to

\[
\sum_{k=1}^{\mathfrak{D}-1}A_{k}=\sum_{\deg N<\mathfrak{D}}\chi\left(N\right)=0.
\]
This concludes the Theorem for the case $\chi$ is ``even''.

Now, if $\chi$ is not ``even'' the functional equation gives

\[
\left|L^{*}\left(u,\chi\right)\right|^{2}=\left(q^{\nicefrac{1}{2}}u\right)^{2\mathfrak{D}}\left|L^{*}\left(\frac{1}{qu},\bar{\chi}\right)\right|^{2}
\]
(since then $\Lambda\left(u,\chi\right)=L^{*}\left(u,\chi\right)$).
So

\begin{eqnarray*}
\sum_{m=0}^{2\mathfrak{D}}u^{m}A_{m} & = & q^{\mathfrak{D}}u^{2\mathfrak{D}}\sum_{n=0}^{2\mathfrak{D}}\left(\frac{1}{qu}\right)^{n}A_{n}=\sum_{n=0}^{2\mathfrak{D}}u^{\left(2\mathfrak{D}-n\right)}q^{\mathfrak{D}-n}A_{n}
\end{eqnarray*}
and by denoting $m=2\mathfrak{D}-n$ we get

\[
=\sum_{m=0}^{2\mathfrak{D}}u^{m}q^{m-\mathfrak{D}}A_{2\mathfrak{D}-m}.
\]
We conclude 

\begin{equation}
A_{m}=q^{m-\mathfrak{D}}A_{2\mathfrak{D}-m}\;\Rightarrow\; A_{2\mathfrak{D}-m}=q^{\mathfrak{D}-m}A_{m}.\label{eq:A_n}
\end{equation}
Now,

\begin{eqnarray*}
\left|L\left(\frac{1}{2},\chi\right)\right|^{2} & =\sum_{n=0}^{2\mathfrak{D}} & q^{-\nicefrac{n}{2}}A_{n}=\sum_{n=0}^{\mathfrak{D}}q^{-\nicefrac{n}{2}}A_{n}+\sum_{n=\mathfrak{D}+1}^{2\mathfrak{D}}q^{-\nicefrac{n}{2}}A_{n}
\end{eqnarray*}
once more, we denote $m=2\mathfrak{D}-n$

\begin{eqnarray*}
 & = & \sum_{n=0}^{\mathfrak{D}}q^{-\nicefrac{n}{2}}A_{n}+\sum_{m=0}^{\mathfrak{D}-1}q^{-\nicefrac{\left(2\mathfrak{D}-m\right)}{2}}A_{2\mathfrak{D}-m}
\end{eqnarray*}
by equation \ref{eq:A_n}

\begin{eqnarray*}
 & = & \sum_{n=0}^{\mathfrak{D}}q^{-\nicefrac{n}{2}}A_{n}+\sum_{m=0}^{\mathfrak{D}-1}q^{-\nicefrac{m}{2}}A_{m}\\
 & = & 2\sum_{n=0}^{\mathfrak{D}-1}q^{-\nicefrac{n}{2}}A_{n}+q^{-\nicefrac{\mathfrak{D}}{2}}A_{\mathfrak{D}}
\end{eqnarray*}
and this conclude the Theorem for the case $\chi$ is not ``even''.\end{proof}
\begin{rem*}
We haven`t used the fact that $Q$ is irreducible, therefore the Theorem
applies to any $Q\in\mathcal{A}$. However, only if $Q$ is an irreducible
polynomial, then every $\chi\neq\chi_{0}$ is a primitive character.
Hence, only then we can conclude
\[
\underset{\chi\pmod Q}{\mathrm{\sum}^{*}}\left|L\left(\frac{1}{2},\chi\right)\right|^{2}=\underset{\chi\pmod Q}{\mathrm{\sum}^{*}}2\left(\sum_{\deg NM<\mathfrak{D}}\frac{\chi\left(N\right)\bar{\chi}\left(M\right)}{\left|NM\right|^{\nicefrac{1}{2}}}+\pi\left(\chi\right)\right).
\]

\end{rem*}

\section{\label{sec:The-Fourth-Moment}The Fourth Moment}

In \cite{key-2}, K. Soundararajan proved that for all large $q$ 

\[
\underset{\chi\pmod q}{\mathrm{\sum}^{*}}\left|L\left(\frac{1}{2},\chi\right)\right|^{4}=\frac{\varphi^{*}\left(q\right)}{2\pi^{2}}\prod_{p\mid q}\frac{\left(1-p^{-1}\right)^{3}}{\left(1+p^{-1}\right)}\left(\log q\right)^{4}+O\left(q\left(\log q\right)^{\frac{7}{2}}\right)
\]
Here $\mathrm{\sum^{*}}$ denotes summation over primitive characters
$\chi\pmod q,\;\varphi^{*}(q)$ denotes the number of primitive characters
modulo $q$, and $\omega\left(q\right)$ denotes the number of distinct
prime factors of $q$. 

We will show the polynomial analog of this Theorem for a prime polynomial
$Q$. 
\begin{thm}
Let $Q$ be a prime polynomial, $\deg Q-1=\mathfrak{D}$. For all
large $\mathfrak{D}\geqslant1$ we have

\[
\frac{1}{\phi\left(Q\right)}\sum_{\chi\neq\chi_{0}}\left|L\left(\frac{1}{2},\chi\right)\right|^{4}=\frac{q-1}{12q}\mathfrak{D}^{4}+O\left(\mathfrak{D}^{3}\right).
\]

\end{thm}
This main Theorem will follow from the following Lemmas:

\subsection{Preliminary Lemmas}
\begin{lem}
\label{lem:sum 2^w}For all integers $x\geq0$

\[
\sum_{\deg N\leq x}\frac{2^{\omega\left(N\right)}}{\left|N\right|}=\frac{\left(q-1\right)x^{2}+\left(3q+1\right)x+2q}{2q}.
\]
\end{lem}
\begin{proof}
For $Re(s)>1$ we can define 

\begin{eqnarray*}
F\left(s\right): & = & \sum_{N\:\mbox{monic}}\frac{2^{\omega\left(N\right)}}{\left|N\right|^{s+1}}=\prod_{P}\left(1+\sum_{k=1}^{\infty}\frac{2^{\omega\left(P^{k}\right)}}{\left|P\right|^{\left(s+1\right)k}}\right)\\
 & = & \prod_{P}\left(1+2\sum_{k=1}^{\infty}\left|P\right|^{-\left(s+1\right)k}\right)=\prod_{P}\left(\frac{1+\left|P\right|^{-\left(s+1\right)}}{1-\left|P\right|^{-\left(s+1\right)}}\right)\\
 & = & \frac{\prod_{P}\left(1-\left|P\right|^{-2(s+1)}\right)}{\prod_{P}\left(1-\left|P\right|^{-(s+1)}\right)^{2}}=\frac{\zeta\left(s+1\right)^{2}}{\zeta\left(2s+2\right)}\\
 & = & \frac{\left(1-q^{-1-2s}\right)}{\left(1-q^{-s}\right)^{2}}
\end{eqnarray*}
where the first equality result of the multiplicativity of the function
$\frac{2^{\omega\left(N\right)}}{\left|N\right|^{s+1}}$ . Now, we
look at $F$ as a function of $T:=q^{-s}$, and define $A_{n}:=q^{-n}\sum_{\deg N=n}2^{\omega\left(N\right)}$,
so $F\left(T\right)=\sum_{k=0}^{\infty}A_{k}T^{k}=\frac{\left(1-q^{-1}T^{2}\right)}{\left(1-T\right)^{2}}$.
Since

\[
\frac{1}{\left(1-T\right)^{2}}=\left(\sum_{k=0}^{\infty}T^{k}\right)^{2}=\sum_{k=0}^{\infty}\left(k+1\right)T^{k}
\]
we have 

\begin{eqnarray*}
F(T) & = & \left(1-q^{-1}T^{2}\right)\sum_{k=0}^{\infty}\left(k+1\right)T^{k}=\sum_{k=0}^{\infty}\left(k+1\right)T^{k}-\sum_{k=0}^{\infty}q^{-1}\left(k+1\right)T^{k+2}\\
 & = & 1+2T+\sum_{k=2}^{\infty}\left(k+1-\frac{k-1}{q}\right)T^{k}=1+\sum_{k=1}^{\infty}\left(k+1-\frac{k-1}{q}\right)T^{k},
\end{eqnarray*}
therefore

\[
A_{k}=\begin{cases}
1; & k=0\\
k+1-\frac{k-1}{q}; & k\geq1,
\end{cases}
\]
hence 

\begin{eqnarray*}
\sum_{\deg N\leq x}\frac{2^{\omega\left(N\right)}}{\left|N\right|} & = & \sum_{k=0}^{x}A_{k}\\
 & = & 1+\sum_{k=1}^{x}\left(k+1-\frac{k-1}{q}\right)\\
 & = & \frac{\left(q-1\right)x^{2}+\left(3q+1\right)x+2q}{2q}.
\end{eqnarray*}
\end{proof}
\begin{lem}
\label{lem:diagonal sum}We have

\[
\underset{\deg AB,\deg CD<\mathfrak{D}}{\underset{AC=BD}{\sum}}\left|ABCD\right|^{-\nicefrac{1}{2}}=\frac{q-1}{48q}\mathfrak{D}^{4}+O\left(\mathfrak{D}^{3}\right).
\]
\end{lem}
\begin{proof}
For $AC=BD$ we can write $A=UR,\: B=US,\: C=VS,\: D=VR$, where $R$
and $S$ are coprime. We put $N=RS$, and note that given $N$ there
are $2^{\omega(N)}$ ways of writing it as $RS$ with $R$ and $S$
coprime. Note also that $AB=U^{2}N$ and $CD=V^{2}N$. So

\begin{eqnarray*}
\underset{\deg AB,\deg CD<\mathfrak{D}}{\underset{AC=BD}{\sum}}\left|ABCD\right|^{-\nicefrac{1}{2}} & = & \sum_{\deg N<\mathfrak{D}}\frac{2^{\omega\left(N\right)}}{\left|N\right|}\left(\underset{\deg U<\nicefrac{1}{2}(\mathfrak{D}-\deg N)}{\sum}\frac{1}{\left|U\right|}\right)^{2}\\
 & = & \sum_{\deg N<\mathfrak{D}}\frac{2^{\omega\left(N\right)}}{\left|N\right|}\left\lceil \frac{\mathfrak{D}-\deg N}{2}\right\rceil ^{2}\\
 & = & \sum_{\deg N<\mathfrak{D}}\frac{2^{\omega\left(N\right)}}{\left|N\right|}\left[\left(\frac{\mathfrak{D}-\deg N}{2}\right)^{2}+O\left(1\right)\right]\\
 & = & \frac{1}{4}\sum_{\deg N<\mathfrak{D}}\frac{2^{\omega(N)}}{\left|N\right|}\left(\mathfrak{D}-\deg N\right)^{2}+O\left(\underset{\deg N<\mathfrak{D}}{\sum}\frac{2^{\omega\left(N\right)}}{\left|N\right|}\right),
\end{eqnarray*}
and by Lemma \ref{lem:sum 2^w} we arrive at 
\[
=\frac{1}{4}\underset{\deg N<\mathfrak{D}}{\sum}\frac{2^{\omega(N)}}{\left|N\right|}\left(\mathfrak{D}-\deg N\right)^{2}+O\left(\mathfrak{D}^{2}\right).
\]
Now, as seen in Lemma \ref{lem:sum 2^w}

\begin{eqnarray*}
\sum_{\deg N<\mathfrak{D}}\frac{2^{\omega\left(N\right)}}{\left|N\right|}\left(\mathfrak{D}-\deg N\right)^{2} & = & \sum_{k=0}^{\mathfrak{D}-1}A_{k}\left(\mathfrak{D}-k\right)^{2}\\
 & = & \mathfrak{D}^{2}+\sum_{k=1}^{\mathfrak{D}-1}\left(k+1-\frac{k-1}{q}\right)\left(\mathfrak{D}-k\right)^{2}\\
 & = & \mathfrak{D}^{2}+\frac{\left(q-1\right)\mathfrak{D}^{4}+4\left(q+1\right)\mathfrak{D}^{3}-\left(7q+5\right)\mathfrak{D}^{2}+2\left(q+1\right)\mathfrak{D}}{12q}\\
 & = & \frac{\left(q-1\right)\mathfrak{D}^{4}+4\left(q+1\right)\mathfrak{D}^{3}+5\left(q-1\right)\mathfrak{D}^{2}+2\left(q+1\right)\mathfrak{D}}{12q},
\end{eqnarray*}
which completes the proof.

\end{proof}
\begin{lem}
\label{lem:diagonal sum>D-3}We have

\[
\underset{\deg CD\leq\mathfrak{D}}{\underset{\mathfrak{D}-3\leq\deg AB\leq\mathfrak{D}}{\underset{AC=BD}{\sum}}}\left|ABCD\right|^{-\nicefrac{1}{2}}\leq\frac{\left(q-1\right)\mathfrak{D}^{3}}{2q}+O\left(\mathfrak{D}^{2}\right).
\]
\end{lem}
\begin{proof}
As seen in Lemma \ref{lem:sum 2^w} and \ref{lem:diagonal sum} we
get 

\begin{eqnarray*}
\underset{\deg CD\leq\mathfrak{D}}{\underset{\mathfrak{D}-3\leq\deg AB\leq\mathfrak{D}}{\underset{AC=BD}{\sum}}}\left|ABCD\right|^{-\nicefrac{1}{2}} & = & \sum_{\deg N\leq\mathfrak{D}}\frac{2^{\omega\left(N\right)}}{\left|N\right|}\left(\sum_{\frac{\mathfrak{D}-3-\deg N}{2}\leq\deg U\leq\frac{\mathfrak{D}-\deg N}{2}}\frac{1}{\left|U\right|}\right)\\
 &  & \qquad\qquad\cdot\left(\sum_{\deg V\leq\frac{\mathfrak{D}-\deg N}{2}}\frac{1}{\left|V\right|}\right)\\
 & \leq & 2\sum_{\deg N\leq\mathfrak{D}}\frac{2^{\omega\left(N\right)}}{\left|N\right|}\left\lceil \frac{\mathfrak{D}-\deg N}{2}+1\right\rceil \\
 & \leq & \sum_{k=0}^{\mathfrak{D}}A_{k}\left(\mathfrak{D}-k+3\right)\\
 & = & \frac{\left(q-1\right)\mathfrak{D}^{3}}{2q}+O\left(\mathfrak{D}^{2}\right).
\end{eqnarray*}
\end{proof}
\begin{lem}
\label{lem:nondiagonal sum}Let $Q$ be an irreducible monic polynomial,
and let $Z_{1}$ and $Z_{2}$ be positive integers, $2<Z_{1},Z_{2}\leq\mathfrak{D}$,
then

\[
\underset{AC\neq BD}{\underset{AC\equiv BD\pmod Q}{\underset{Z_{2}-2\leq\deg CD\leq Z_{2}}{\underset{Z_{1}-2\leq\deg AB\leq Z_{1}}{\sum}}}}1\ll\frac{q^{Z_{1}+Z_{2}}\left(Z_{1}Z_{2}\right)^{3}}{\left|Q\right|}+\mathfrak{D}.
\]
\end{lem}
\begin{proof}
If $Z_{1}+Z_{2}\leq\mathfrak{D}$ we get $\deg ABCD\leq\mathfrak{D}$
so $AC\equiv BD\pmod Q$ iff $AC=BD$ so the sum is zero. Hence we
can assume $Z_{1}+Z_{2}>\mathfrak{D}$. By symmetry we may just focus
on the terms with $\deg AC>\deg BD$. Denote $U=BD$ and $W=AC$.
Note that since $\deg AC\geq\deg Q=\mathfrak{D}+1$ and $\deg U=\deg BD\leq Z_{1}+Z_{2}$
we also have $\mathfrak{D}<\deg W=\deg AC\leq Z_{1}+Z_{2}-\deg U$.
Hence

\begin{eqnarray*}
 & \ll & \sum_{\deg U\leq Z_{1}+Z_{2}-\mathfrak{D}}d_{\left(\mathfrak{D}\right)}\left(U\right)\underset{W\equiv U\pmod Q}{\underset{\mathfrak{D}<\deg W\leq Z_{1}+Z_{2}-\deg U}{\sum}}d_{\left(\mathfrak{D}\right)}\left(W\right).
\end{eqnarray*}
 Now, we divide this sum into the following: 
\begin{eqnarray*}
\mathrm{I} & = & \sum_{\deg U\leq Z_{1}+Z_{2}-\mathfrak{D}}d\left(U\right)\underset{W\equiv U\pmod Q}{\underset{\deg W=\mathfrak{D}+1}{\sum}}d\left(W\right)\\
\mathrm{II} & = & \sum_{\deg U\leq Z_{1}+Z_{2}-\mathfrak{D}}d\left(U\right)\underset{W\equiv U\pmod Q}{\underset{\mathfrak{D}+2\leq\deg W\leq Z_{1}+Z_{2}-\deg U}{\sum}}d\left(W\right).
\end{eqnarray*}
First, since $U\neq W$ and $\deg W>\deg U$ we know 
\begin{eqnarray*}
\mathrm{I} & = & \sum_{\deg U=1}d\left(U\right)\underset{W\equiv U\pmod Q}{\underset{\deg W=\mathfrak{D}+1}{\sum}}d\left(W\right)\\
 & = & \sum_{\deg U=1}d\left(U+Q\right),
\end{eqnarray*}
when the second equation is due to the fact $W$ is a monic polynomial.
We can use the bound $d\left(N\right)\ll\deg N$ and obtain 
\[
\ll q\cdot\mathfrak{D}\ll\mathfrak{D}.
\]
In the second sum, for $\alpha=\frac{1}{2\mathfrak{D}}$ we have $\left|W\right|^{\left(1-\alpha\right)}\geq q^{\left(\mathfrak{D}+2\right)\left(1-\frac{1}{2\mathfrak{D}}\right)}>q^{\mathfrak{D}+1\frac{1}{2}-\frac{1}{2\mathfrak{D}}}\geq\left|Q\right|$,
so we can use the bound $\underset{N\equiv A\pmod Q}{\underset{\deg N\leq x}{\sum}}d(N)\ll\frac{q^{x}x}{\phi(Q)}$,
Theorem \ref{thm:divisor sum},%
{} which yields

\[
\mathrm{II}\ll\underset{\deg U\leq Z_{1}+Z_{2}-\mathfrak{D}}{\sum}d(U)\frac{q^{Z_{1}+Z_{2}}\left(Z_{1}+Z_{2}\right)}{\left|Q\right|\left|U\right|}\ll\frac{q^{Z_{1}+Z_{2}}}{\left|Q\right|}\left(Z_{1}+Z_{2}\right)^{3}
\]
by Lemma \ref{lem: divisor sum}. This completes the proof.\end{proof}
\begin{cor}
\label{cor:nondiagonal }We have

\[
\underset{\deg AB,\deg CD\leq\mathfrak{D}}{\underset{AC\neq BD}{\underset{AC\equiv BD}{\sum}}}\left|ABCD\right|^{-\nicefrac{1}{2}}\ll\mathfrak{D}^{3}.
\]
\end{cor}
\begin{proof}
To estimate this sum we divide the terms $\deg AB,\deg CD\leq\mathfrak{D}$
into dyadic blocks. Consider the block $Z_{1}-2\leq\deg AB<Z_{1}$,
and $Z_{2}-2\leq\deg CD<Z_{2}$. By Lemma \ref{lem:nondiagonal sum}
the contribution of each block to the sum is

\[
\ll q^{-\nicefrac{1}{2}\left(Z_{1}+Z_{2}\right)}\left[\frac{q^{Z_{1}+Z_{2}}}{\left|Q\right|}\left(Z_{1}+Z_{2}\right)^{3}+\mathfrak{D}\right]\ll\frac{q^{\nicefrac{1}{2}\left(Z_{1}+Z_{2}\right)}}{\left|Q\right|}\mathfrak{D}^{3}
\]
Summing over all such dyadic blocks we obtain that 

\[
\underset{\deg AB,\deg CD\leq\mathfrak{D}}{\underset{AC\neq BD}{\underset{AC\equiv BD}{\sum}}}\left|ABCD\right|^{-\nicefrac{1}{2}}\ll\frac{q^{\mathfrak{D}}}{\left|Q\right|}\mathfrak{D}^{3}\ll\mathfrak{D}^{3}.
\]
This proves the Corollary.
\end{proof}

\subsection{Proof of the Main Theorem}

By the previous section, Proposition \ref{thm:squared L}, we have

\[
\left|L\left(\frac{1}{2},\chi\right)\right|^{2}=2\sum_{\deg AB<\mathfrak{D}}\frac{\chi\left(A\right)\bar{\chi}\left(B\right)}{\left|AB\right|^{\nicefrac{1}{2}}}+\pi\left(\chi\right)
\]
when
\[
\pi\left(\chi\right):=\begin{cases}
\begin{array}{l}
2q^{-\nicefrac{\mathfrak{D}}{2}}\left(A_{\mathfrak{D}-3}-A_{\mathfrak{D}-2}+A_{\mathfrak{D}-1}\right)\\
\qquad-\left(q^{-\nicefrac{\left(\mathfrak{D}-1\right)}{2}}-q^{-\nicefrac{\left(\mathfrak{D}+1\right)}{2}}\right)A_{\mathfrak{D}-2};
\end{array} & \chi\mbox{ "even"}\\
q^{-\nicefrac{\mathfrak{D}}{2}}A_{\mathfrak{D}}; & \mbox{else}
\end{cases}
\]
therefore, we get

\begin{eqnarray*}
\frac{1}{\phi\left(Q\right)}\sum_{\chi\neq\chi_{0}}\left|L\left(\frac{1}{2},\chi\right)\right|^{4} & = & \frac{1}{\phi\left(Q\right)}\sum_{\chi\neq\chi_{0}}[4\sum_{\deg AB,\deg CD<\mathfrak{D}}\frac{\chi\left(AC\right)\bar{\chi}\left(BD\right)}{\left|ABCD\right|^{\nicefrac{1}{2}}}\\
 &  & +4\pi\left(\chi\right)\sum_{\deg AB<\mathfrak{D}}\frac{\chi\left(A\right)\bar{\chi}\left(B\right)}{\left|AB\right|^{\nicefrac{1}{2}}}+\pi\left(\chi\right)^{2}]
\end{eqnarray*}
and we can divide this sum into the following:

\begin{eqnarray*}
\mathrm{I} & := & \frac{4}{\phi\left(Q\right)}\sum_{\chi\neq\chi_{0}}\sum_{\deg AB,\deg CD<\mathfrak{D}}\frac{\chi\left(AC\right)\bar{\chi}\left(BD\right)}{\left|ABCD\right|^{\nicefrac{1}{2}}}\\
\mathrm{II} & := & \frac{4}{\phi\left(Q\right)}\sum_{\chi\neq\chi_{0}}\pi\left(\chi\right)\underset{\deg AB<\mathfrak{D}}{\sum}\frac{\chi\left(A\right)\bar{\chi}\left(B\right)}{\left|AB\right|^{\nicefrac{1}{2}}}\\
\mathrm{III} & := & \frac{1}{\phi\left(Q\right)}\sum_{\chi\neq\chi_{0}}\pi\left(\chi\right)^{2}.
\end{eqnarray*}

First, since 

\[
\sum_{\chi\neq\chi_{0}}\chi\left(AC\right)\bar{\chi}\left(BD\right)=\begin{cases}
\phi\left(Q\right)-1, & AC\equiv BD\pmod Q\\
-1, & \mbox{else}
\end{cases}
\]
and 

\begin{eqnarray*}
\frac{1}{\phi\left(Q\right)}\sum_{\deg AB,\deg CD<\mathfrak{D}}\left|ABCD\right|^{-\nicefrac{1}{2}} & = & \frac{1}{\phi\left(Q\right)}\left(\sum_{\deg AB<\mathfrak{D}}\left|AB\right|^{-\nicefrac{1}{2}}\right)^{2}\\
 & \leq & \frac{1}{\phi\left(Q\right)}\left(\sum_{n=0}^{\mathfrak{D}-1}q^{-\nicefrac{n}{2}}q^{n}\sum_{m=0}^{\mathfrak{D}-1-n}q^{-\nicefrac{m}{2}}q^{m}\right)^{2}\\
 & \ll & \frac{1}{\left|Q\right|}q^{\mathfrak{D}}\mathfrak{D}^{2}\ll\mathfrak{D}^{2}
\end{eqnarray*}
we obtain 

\begin{eqnarray*}
\frac{1}{\phi\left(Q\right)}\sum_{\deg AB,\deg CD<\mathfrak{D}}\sum_{\chi\neq\chi_{0}}\frac{\chi\left(AC\right)\bar{\chi}\left(BD\right)}{\left|ABCD\right|^{\nicefrac{1}{2}}} & = & \sum_{AC\equiv BD}\left|ABCD\right|^{-\nicefrac{1}{2}}\\
 &  & \quad-\frac{1}{\phi\left(Q\right)}\sum_{\deg AB,\deg CD<\mathfrak{D}}\left|ABCD\right|^{-\nicefrac{1}{2}}\\
 & = & \sum_{AC=BD}\left|ABCD\right|^{-\nicefrac{1}{2}}\\
 &  & \quad+\underset{AC\neq BD}{\underset{AC\equiv BD,}{\sum}}\left|ABCD\right|^{-\nicefrac{1}{2}}+O\left(\mathfrak{D}^{2}\right).
\end{eqnarray*}
Hence in the first case we have

\[
\mathrm{I}=4\left[\underset{\deg AB,\deg CD<\mathfrak{D}}{\underset{AC=BD}{\sum}}\left|ABCD\right|^{-\nicefrac{1}{2}}+\underset{\deg AB,\deg CD<\mathfrak{D}}{\underset{AC\neq BD}{\underset{AC\equiv BD}{\sum}}}\left|ABCD\right|^{-\nicefrac{1}{2}}+O\left(\mathfrak{D}^{2}\right)\right]
\]
From Lemma \ref{lem:diagonal sum} and Corollary to Lemma \ref{lem:nondiagonal sum}
we conclude

\[
\mathrm{I}=\frac{q-1}{12q}\mathfrak{D}^{4}+O\left(\mathfrak{D}^{3}\right).
\]

Next, accordingly, from $\underset{\chi\neq\chi_{0}}{\sum}\pi\left(\chi\right)\ll\phi\left(Q\right)\underset{\mathfrak{D}-3\leq\deg AB\leq\mathfrak{D}}{\underset{A=B}{\sum}}\left|AB\right|^{-\nicefrac{1}{2}}$
we have

\[
\mathrm{II}\ll4\left[\underset{\deg CD<\mathfrak{D}}{\underset{\mathfrak{D}-3\leq\deg AB\leq\mathfrak{D}}{\underset{AC=BD}{\sum}}}\left|ABCD\right|^{-\nicefrac{1}{2}}+\underset{\deg CD<\mathfrak{D}}{\underset{\mathfrak{D}-3\leq\deg AB\leq\mathfrak{D}}{\underset{AC\neq BD}{\underset{AC\equiv BD}{\sum}}}}\left|ABCD\right|^{-\nicefrac{1}{2}}+O\left(\mathfrak{D}^{2}\right)\right]
\]
From Lemma \ref{lem:diagonal sum>D-3} and Corollary \ref{cor:nondiagonal }
we have

\[
\mathrm{II}\ll\mathfrak{D}^{3}
\]

Lastly, from $\underset{\chi\neq\chi_{0}}{\sum}\pi\left(\chi\right)\ll\phi\left(Q\right)\underset{\mathfrak{D}-3\leq\deg AB\leq\mathfrak{D}}{\underset{A=B}{\sum}}\left|AB\right|^{-\nicefrac{1}{2}}$
we also receive

\[
\mathrm{III}\ll\mathrm{II}\ll\mathfrak{D}^{3}
\]
This proves the Theorem.

\section{Lower Bound}

Here we will show the second main result of this paper, an analog
of the general lower bound of Rudnick and Soundararajan \cite{key-11},
for the $2k$-th moment
\begin{thm}
Let $k$ be a fixed natural number. Then for all irreducible polynomial
$Q$, with a sufficiently large degree 
\[
\underset{\chi\neq\chi_{0}}{\sum_{\chi\pmod Q}}\left|L\left(\frac{1}{2},\chi\right)\right|^{2k}\gg_{k}\left|Q\right|\left(\deg Q\right)^{k^{2}}.
\]

\end{thm}
We shall require the following Lemmas in order to prove the Theorem
above: %

\begin{lem}
For $k\in\mathbb{N}$
\[
\sum_{\deg N\geq0}\frac{d_{k}\left(N\right)^{2}}{\left|N\right|^{s}}=\zeta(s)^{k^{2}}\prod_{P}p_{k}\left(\left|P\right|^{-s}\right)
\]
when 
\[
p_{k}\left(x\right)=\left(1-x\right)^{\left(k-1\right)^{2}}\sum_{n=0}^{k-1}\binom{k-1}{n}^{2}x^{n}.
\]
\end{lem}
\begin{proof}
We have
\[
\sum_{\deg N\geq0}\frac{d_{k}\left(N\right)^{2}}{\left|N\right|^{s}}=\prod_{P}\left(\sum_{l=0}^{\infty}\frac{\binom{k+l-1}{k-1}^{2}}{\left|P\right|^{ls}}\right).
\]
If we denote 
\[
f_{k}\left(P^{-s}\right)=\sum_{l=0}^{\infty}\frac{\binom{k+l-1}{k-1}^{2}}{\left|P\right|^{ls}}=1+\frac{k^{2}}{\left|P\right|^{s}}+\dots,
\]
we have 
\begin{eqnarray*}
\left(1-\left|P\right|^{-s}\right)^{k^{2}}f_{k}\left(P^{-s}\right) & = & \left(1-\frac{k^{2}}{\left|P\right|^{s}}+\dots\right)\left(1+\frac{k^{2}}{\left|P\right|^{s}}+\dots\right)\\
 & = & 1+O\left(\left|P\right|^{-2s}\right)
\end{eqnarray*}
hence 
\[
\sum_{\deg N\geq0}\frac{d_{k}\left(N\right)^{2}}{\left|N\right|^{s}}=\zeta\left(s\right)^{k^{2}}\prod_{p}p_{k}\left(\left|P\right|^{-s}\right)
\]
when 
\begin{eqnarray*}
p_{k}(x) & = & \left(1-x\right)^{k^{2}}\left(\sum_{l=0}^{\infty}\binom{k+l-1}{k-1}^{2}x^{l}\right)\\
 & = & \left(1-x\right)^{\left(k-1\right)^{2}}\left(1-x\right)^{2k-1}\left(\sum_{l=0}^{\infty}\binom{k+l-1}{k-1}^{2}x^{l}\right)\\
 & = & \left(1-x\right)^{\left(k-1\right)^{2}}\left(\sum_{n=0}^{2k-1}\binom{2k-1}{n}\left(-x\right)^{n}\right)\left(\sum_{l=0}^{\infty}\binom{k+l-1}{k-1}^{2}x^{l}\right)
\end{eqnarray*}
for any $m\leq k-1$ the coefficient of $x^{m}$ in the sum is 
\[
\sum_{d\leq m}\left(-1\right)^{d}\binom{2k-1}{d}\binom{k+m-d-1}{m-d}^{2}=\binom{k-1}{m}^{2},
\]
 otherwise it is $0$. \end{proof}
\begin{cor}
\label{cor:squar divisor sum}We have
\begin{eqnarray*}
\sum_{\deg N\geq0}\frac{d_{k}\left(N\right)^{2}}{\left|N\right|^{s}} & = & \zeta\left(s\right)^{k^{2}}c\left(s\right)\\
 & = & \frac{c\left(s\right)}{\left(1-q^{-s+1}\right)^{k^{2}}}
\end{eqnarray*}
when $c(s)=\prod_{P}p_{k}(\left|P\right|^{-s})$ is an analytic function
when $\Re(s)>\nicefrac{1}{2}$.\end{cor}
\begin{lem}
\label{lem:lower bound dk^2 sum}For fixed $k\in\mathbb{N}$

\[
\sum_{\deg N\leq y}\frac{d_{k}\left(N\right)^{2}}{\left|N\right|}\sim c_{k}y^{k^{2}}.
\]
\end{lem}
\begin{proof}
For $k\geq1$, define 
\[
A_{n}:=q^{-n}\sum_{\deg N=n}d_{k}\left(N\right)^{2}.
\]
We will show that 
\begin{equation}
A_{n}=h_{k^{2}-1}\left(n\right)+O\left(q^{-\delta n}\right)=c_{k^{2}-1}n^{k^{2}-1}+O_{k}\left(n^{k^{2}-2}\right)\label{eq:asymp for An}
\end{equation}
 where $h_{l}\left(x\right)$ is a polynomial of degree $l$ in $x$,
with leading coefficient $c_{k^{2}}>0$, and $0<\delta<1/2$. Consequently,
we find that 
\[
\sum_{\deg N\leq y}\frac{d_{k}(N)^{2}}{|N|}=c_{k}y^{k^{2}}+O(y^{k^{2}-1})
\]
 with $c_{k}>0$. 

If we look at the functions in Corollary \ref{cor:squar divisor sum}
as functions of $T:=q^{-s}$, we get
\begin{eqnarray*}
F\left(T\right):=\sum_{n\geq0}A_{n}T^{n} & = & \frac{C\left(T\right)}{\left(1-T\right)^{k^{2}}},
\end{eqnarray*}
where $C(T)=\prod_{P}p_{k}\left(\left|P\right|^{-1}T^{\deg P}\right)$
is an analytic function when $T<q^{\nicefrac{1}{2}}$. Since the product
is convergent (hence nonzero) in $T<q^{\nicefrac{1}{2}}$, by Cauchy's
integral formula we get
\[
A_{n}=\frac{1}{2\pi i}\ointop_{\left|T\right|=\nicefrac{1}{2}}\frac{F\left(T\right)}{T^{n+1}}dT,
\]
where the contour of integration is a small circle about the origin,
containing no singularity of the integrand except $T=0$, traversed
anti-clockwise. Expanding the contour beyond the pole of $F\left(T\right)$
at $T=1$ gives
\[
A_{n}=-\mbox{Res}_{T=1}\frac{C\left(T\right)}{\left(1-T\right)^{k^{2}}T^{n+1}}+\frac{1}{2\pi i}\oint_{\left|T\right|=q^{\delta}}\frac{F\left(T\right)}{T^{n+1}}dT
\]
where the contour of integration is the circle $\left|T\right|=q^{\delta}$
traversed anti-clockwise, and $0<\delta<\nicefrac{1}{2}$. We have
\[
-\mbox{Res}_{T=1}\frac{C\left(T\right)}{\left(1-T\right)^{k^{2}}T^{n+1}}=-\frac{\left(-1\right)^{k^{2}}}{\left(k^{2}-1\right)!}\frac{d^{k^{2}-1}}{dT^{k^{2}-1}}\frac{\tilde{c}\left(T\right)}{T^{n+1}}\Big|_{T=1}.
\]
 By Leibnitz's rule, taking into account that the $\ensuremath{r}$-th
derivative of $\nicefrac{1}{T^{n+1}}$ at $T=1$ is $\left(-1\right)^{r}\left(n+1\right)\cdot\dots\cdot\left(n+1-\left(r-1\right)\right)$,
\begin{eqnarray*}
-\frac{\left(-1\right)^{k^{2}}}{\left(k^{2}-1\right)!}\frac{d^{k^{2}-1}}{dT^{k^{2}-1}}\frac{C\left(T\right)}{T^{n+1}}\Big|_{T=1} & = & \frac{\left(-1\right)^{k^{2}-1}}{\left(k^{2}-1\right)!}\sum_{r=0}^{k^{2}-1}\binom{k^{2}-1}{r}C^{\left(k^{2}-1-r\right)}\left(1\right)\cdot\\
 &  & \qquad\qquad\left(-1\right)^{r}\left(n+1\right)\cdots\left(n+1-\left(r-1\right)\right)\\
 & = & \sum_{r=0}^{k^{2}-1}\binom{n+1}{r}\frac{\left(-1\right)^{k^{2}-1-r}}{\left(k^{2}-1-r\right)!}C^{\left(k^{2}-1-r\right)}\left(1\right)=:h_{k^{2}-1}\left(n\right)
\end{eqnarray*}
 is a polynomial of degree $k^{2}-1$ in $n$, whose leading coefficient
comes from the term $r=k^{2}-1$: 
\[
h_{k^{2}-1}\left(n\right)=\frac{C\left(1\right)}{\left(k^{2}-1\right)!}n^{k^{2}-1}+O\left(n^{k^{2}-2}\right).
\]
Here $C\left(1\right)=\prod_{P}p_{k}\left(\left|P\right|^{-2}\right)$
with 
\[
p_{k}(x)=\left(1-x\right)^{\left(k-1\right)^{2}}\sum_{n=0}^{k-1}\binom{k-1}{n}^{2}x^{n}
\]
 so clearly $C\left(1\right)>0$. Finally, we bound the integral over
the contour $|T|=q^{\delta}$ by 
\[
\left|\frac{1}{2\pi i}\oint_{\left|T\right|=q^{\delta}}\frac{F\left(T\right)}{T^{n+1}}dT\right|\leq q^{-n\delta}\max_{|T|=q^{\delta}}|F\left(T\right)|
\]
 which proves \ref{eq:asymp for An}, and hence the Theorem. 

\end{proof}
\begin{thm}
Let $k$ be a fixed natural number. Then for all irreducible polynomial
$Q$, with a big enough degree
\[
\underset{\chi\neq\chi_{0}}{\sum_{\chi\pmod Q}}\left|L\left(\frac{1}{2},\chi\right)\right|^{2k}\gg_{k}\left|Q\right|\left(\deg Q\right)^{k^{2}}.
\]
\end{thm}
\begin{proof}
Let $x:=\frac{\mathfrak{D}}{2k}$, $A\left(\chi\right):=\sum_{\deg N\leq x}\frac{\chi\left(N\right)}{\left|N\right|}$,
and set 
\[
S_{1}:=\underset{\chi\neq\chi_{0}}{\sum_{\chi\pmod Q}}L\left(\frac{1}{2},\chi\right)A\left(\chi\right)^{k-1}\overline{A\left(\chi\right)}^{k}\:,\qquad S_{2}:=\underset{\chi\neq\chi_{0}}{\sum_{\chi\pmod Q}}\left|A\left(\chi\right)\right|^{2k}.
\]
By the Triangle and Hölder's inequalities we have
\begin{eqnarray*}
\left|\underset{\chi\neq\chi_{0}}{\sum_{\chi\pmod Q}}L\left(\frac{1}{2},\chi\right)A\left(\chi\right)^{k-1}\overline{A\left(\chi\right)}^{k}\right| & \leq & \underset{\chi\neq\chi_{0}}{\sum_{\chi\pmod Q}}\left|L\left(\frac{1}{2},\chi\right)\right|\left|A\left(\chi\right)\right|{}^{2k-1}\\
 & \leq & \left(\underset{\chi\neq\chi_{0}}{\sum_{\chi\pmod Q}}\left|L\left(\frac{1}{2},\chi\right)\right|^{2k}\right)^{\frac{1}{2k}}\\
 &  & \qquad\cdot\left(\underset{\chi\neq\chi_{0}}{\sum_{\chi\pmod Q}}\left|A\left(\chi\right)\right|^{2k}\right)^{\frac{2k-1}{2k}}
\end{eqnarray*}
therefore
\[
\underset{\chi\neq\chi_{0}}{\sum_{\chi\pmod Q}}\left|L\left(\frac{1}{2},\chi\right)\right|^{2k}\geq\frac{\left|S_{1}\right|^{2k}}{S_{2}^{2k-1}}.
\]
Hence, if we show $S_{2}\ll\left|Q\right|\left(\deg Q\right)^{k^{2}}\ll S_{1}$,
the Theorem will follow. 

First we evaluate $S_{2}$. Since $A\left(\chi_{0}\right)\ll q^{\frac{x}{2}}$
we have
\begin{eqnarray*}
S_{2} & = & \sum_{\chi\pmod Q}\left|A\left(\chi\right)\right|^{2k}+O\left(q^{\nicefrac{x}{2}}\right)\\
 & = & \sum_{\deg M,\deg N\leq kx}\frac{d_{k}\left(M,x\right)d_{k}\left(N,x\right)}{\left|MN\right|^{\nicefrac{1}{2}}}\sum_{\chi\pmod Q}\chi\left(M\right)\overline{\chi}\left(N\right)+O\left(q^{\nicefrac{x}{2}}\right).
\end{eqnarray*}
Since $kx=\frac{\mathfrak{D}}{2}<\mathfrak{D}$ the orthogonality
relation for characters mod $Q$ gives that only the diagonal term
$M=N$ survives. Thus, 
\[
S_{2}=\phi(Q)\sum_{\deg N\leq kx}\frac{d_{k}(N,x)^{2}}{\left|N\right|}+O(q^{\nicefrac{x}{2}})\leq\phi(Q)\sum_{\deg N\leq kx}\frac{d_{k}(N)^{2}}{\left|N\right|}+O\left(q^{\nicefrac{x}{2}}\right),
\]
since $d_{k}(N,x)\leq d_{k}(N)$. Due to Lemma \ref{lem:lower bound dk^2 sum},
we get $S_{2}\ll\left|Q\right|\left(\deg Q\right)^{k^{2}}$, as claimed.

Now we evaluate $S_{1}$. We have
\begin{eqnarray*}
S_{1} & = & \underset{\chi\neq\chi_{0}}{\sum_{\chi\pmod Q}}\sum_{\deg N\leq\mathfrak{D}}\frac{\chi\left(N\right)}{\left|N\right|^{\nicefrac{1}{2}}}A\left(\chi\right)^{k-1}\overline{A\left(\chi\right)}^{k}\\
 & = & \sum_{\chi\pmod Q}\sum_{\deg N\leq\mathfrak{D}}\frac{\chi\left(N\right)}{\left|N\right|^{\nicefrac{1}{2}}}A\left(\chi\right)^{k-1}\overline{A\left(\chi\right)}^{k}+O\left(q^{\nicefrac{3\mathfrak{D}}{2}}\right)\\
 & = & \phi\left(Q\right)\sum_{\deg A\leq(k-1)x}\sum_{\deg B\leq kx}\underset{AN\equiv B\pmod Q}{\sum_{\deg N\leq\mathfrak{D}}}\frac{d_{k-1}\left(A,x\right)d_{k}\left(B,x\right)}{\left|ABN\right|^{\nicefrac{1}{2}}}+O\left(\frac{q^{\mathfrak{D}}}{\mathfrak{D}}\right)
\end{eqnarray*}
when the first equality holds according to 
\[
\sum_{\deg N\leq\mathfrak{D}}\frac{1}{\left|N\right|^{\nicefrac{1}{2}}}A(\chi_{0})^{2k-1}\ll q^{\nicefrac{\mathfrak{D}}{2}}q^{(2k-1)\nicefrac{x}{2}}\ll\frac{q^{\mathfrak{D}}}{\mathfrak{D}}
\]
 and the second equality holds by using the orthogonality relation
for characters. Since $\frac{d_{k-1}\left(A,x\right)d_{k}\left(B,x\right)}{\left|ABN\right|^{\nicefrac{1}{2}}}\geq0$
we can write 
\[
S_{1}\geq\phi\left(Q\right)\sum_{\deg B\leq kx}\underset{AN=B}{\sum_{\deg A\leq(k-1)x\:\deg N\leq\mathfrak{D}}}\frac{d_{k-1}\left(A,x\right)d_{k}\left(B,x\right)}{\left|ABN\right|^{\nicefrac{1}{2}}}+O\left(\frac{q^{\mathfrak{D}}}{\mathfrak{D}}\right).
\]

Since 
\[
\underset{AN=B}{\sum_{\deg A\leq(k-1)x,\:\deg N\leq\mathfrak{D}}}d_{k-1}\left(A,x\right)\geq\underset{AN=B}{\sum_{\deg A\leq(k-1)x,\:\deg N\leq x}}d_{k-1}\left(A,x\right)=d_{k}\left(B,x\right)
\]
and $d_{k}(B,x)=d_{k}(B)$ for $\deg B\leq x$, we deduce that
\[
S_{1}\geq\phi\left(Q\right)\sum_{\deg B\leq x}\frac{d_{k}\left(B\right)^{2}}{\left|B\right|}+O\left(\frac{q^{\mathfrak{D}}}{\mathfrak{D}}\right)\gg\left|Q\right|\left(\deg Q\right)^{k^{2}}.
\]
This proves the Theorem. \end{proof}

\end{document}